\documentclass[a4paper]{amsart}
\usepackage{amsmath, amsfonts, amsthm, amssymb}
\usepackage{setspace}
\usepackage{tikz}
\usepackage{xcolor}
\usepackage{xypic}
\usepackage[pdftex,colorlinks,linkcolor=red!66!black,citecolor=blue!66!black]{hyperref}

\newcommand{\RR}{\mathbb{R}}
\newcommand{\II}{\mathbb{I}}
\newcommand{\JJ}{\mathbb{J}}
\newcommand{\NN}{\mathbb{N}}
\newcommand{\ZZ}{\mathbb{Z}}
\newcommand{\bou}{\boldsymbol{\mathrm{b}}}
\newcommand{\e}{\varepsilon}
\newcommand{\Db}{\mathcal{D}^b}
\newcommand{\DbAI}{\mathcal{D}^b(A_{\II})}
\newcommand{\upbnd}{\partial_{\uparrow}}
\newcommand{\dwnbnd}{\partial_{\downarrow}}
\newcommand{\codep}{\mathsf{C}}
\renewcommand{\S}{\mathfrak{S}}
\newcommand{\li}{\langle i\rangle}
\newcommand{\preR}{\Lambda_{\RR}}
\newcommand{\preI}{\Lambda_{\II}}
\newcommand{\preJ}{\Lambda_{\JJ}}
\newcommand{\boldell}{\boldsymbol{\ell}}
\newcommand{\ReppreI}{\mathrm{Rep}(\preI)}
\newcommand{\sawtooth}{\partial|_{[a,b]}}
\newcommand{\kvec}{\Bbbk\text{-}\mathrm{vec}}
\newcommand{\kVec}{\Bbbk\text{-}\mathrm{Vec}}

\newtheorem{thm}{Theorem}

\newtheorem{theorem}{Theorem}[section]
\newtheorem{proposition}[theorem]{Proposition}
\newtheorem{corollary}[theorem]{Corollary}
\newtheorem{lemma}[theorem]{Lemma}
\newtheorem*{conjecture}{Conjecture}
\theoremstyle{definition}
\newtheorem{definition}[theorem]{Definition}
\newtheorem{remark}[theorem]{Remark}
\newtheorem{example}[theorem]{Example}
\newtheorem{notation}[theorem]{Notation}

\newcommand{\Int}{\operatorname{Int}}
\newcommand{\Tor}{\operatorname{Tor}}
\renewcommand{\S}{\mathfrak{S}}
\newcommand{\id}{\textrm{id}}
\newcommand{\unif}{\textrm{unif}}

\DeclareMathOperator{\End}{\mathrm{End}}
\DeclareMathOperator{\Hom}{\mathrm{Hom}}
\DeclareMathOperator{\repp}{\mathrm{rep}}
\DeclareMathOperator{\Rep}{\mathrm{Rep}}

\DeclareMathOperator{\rfp}{\mathrm{rep}^{\text{fp}}}
\DeclareMathOperator{\coker}{\mathrm{coker}}
\DeclareMathOperator{\gen}{\mathrm{gen}}
\DeclareMathOperator{\supp}{\mathrm{supp}}
\DeclareMathOperator{\len}{\mathrm{len}}
\DeclareMathOperator{\lub}{\mathrm{lub}}

\author{J.~D.~Rock}
\address{JDR: Algebra Group, Department of Mathematics, KU Leuven, Leuven, Belgium.
\newline
\indent
Department of Mathematics W16, UGent, Ghent, Belgium}
\email{jobdaisie.rock@kuleuven.be}
\author{H.~Thomas}
\address{HT: CP 8888, Succursale Centre-ville, Montréal QC, Canada H3C 3P8}
\email{thomas.hugh\textunderscore{}r@uqam.ca}
\title{Preprojective categories of type $A$}
\date{\textcolor{red}{\today}}

\begin{document}
	
	\begin{abstract}
		We introduce a continuous version of preprojective algebras of type $A$.
		In particular, we are interested in the preprojective category over an open, bounded subinterval $\II$ of $\RR$, denoted $\preI$.
		We study the representable projective modules and define a useful type of sub- and quotient module called decorous modules.
		These are completely described by a function from the closure $\overline{\II}$ of $\II$ to $\RR$ whose `slopes' are not too steep anywhere.
		We later use these to describe permuton ideals, a generalization of the support $\tau$-tilting ideals of preprojective algebras of type $A_n$, which we call permutation ideals. 
		Once we have our generalization, we show that permutation ideals can be recovered from permuton ideals.
		Moreover, permutation ideals are $\tau$-rigid and we show an analogous property for our permuton ideals.
		Along the way, we classify all the brick $\preI$-modules.
	\end{abstract}
	
	\maketitle
	
	\section{Introduction}
	
	\subsection{Context and goal}		
        There is a generalization of the representation theory of a linearly oriented $A_n$ quiver which replaces the discrete quiver by an interval $\mathbb I$ of the real line. (For simplicity, we will restrict our discussion here to the case of open, though not necessarily bounded, intervals.)  A representation $M$ of $\mathbb I$  associates a finite-dimensional vector space $M(x)$ to each point $x\in\mathbb I$. For any $x_1<x_2$ in $\mathbb I$, there is an associated map from $M(x_2)$ to $M(x_1)$; these maps satisfy certain natural compatibility relations. The representation theory of such intervals is now well-understood \cite{BCB20,R20,IRT23}.  Familiar features of the $A_n$ case reappear in this setting. In particular, the finitely presentable indecomposable representations correspond to subintervals $(a,b]$ of $\mathbb I$. The representation theory of intervals turns out to be relevant to applications in topological data analysis.
        See, for example, \cite{Obook}. 

          In a similar spirit, one might ask about continuous versions of other finite-dimensional algebras. In general, it is not clear how this should be approached. 
          One approach to this, via thread quivers, is studied in \cite{PRY}.
          In the present paper, we consider the problem of defining a continuous version of the preprojective algebras of type $A$. Similarly to the case of the continuous linearly oriented type $A$ quiver, we start with an open interval $\mathbb I$ of the real line. A representation $M$ associates to each $x\in \mathbb I$ a finite-dimensional vector space. For $x_1<x_2$, there is a map from $M(x_1)$ to $M(x_2)$, and a map from $M(x_2)$ to $M(x_1)$, subject to certain natural relations which we shall explain shortly.

          Since the preprojective algebra of type $A_n$ is wild for $n\geq 6$ \cite{ES}, it would be unreasonable to hope to give an explicit description of the full module category of the continuous preprojective algebra. Instead, we shall describe an interesting class of modules. Recall that the radical filtation of an indecomposable projective module of type $A_n$ has the appearance of a (discrete) diamond. For our continuous preprojective algebra, the representable indecomposable projectives look like (continous) diamonds. We will primarily focus our attention their subquotients, which still admit a two-dimensional description. We call these \emph{sheet modules}. We think of them as corresponding to a diamond with a top part and a bottom part removed. 

          The basic support $\tau$-tilting modules of a preprojective algebra of type $A_n$ form a collection of two-sided ideals, each indexed by a permutation $w\in S_{n+1}$, the Weyl group of type $A_n$ \cite{Miz}. We refer to these as \emph{permutation ideals}. 
          We consider an analogous construction for the continuous preprojective algebra. The suitable indexing objects turn out to be a generalization of permutations known as \emph{permutons} \cite{HK,GGKK}; we call the corresponding objects \emph{permuton ideals}.  
          
	\subsection{Organization and Contributions}
	In Section~\ref{sec:finite} we recall the relevant notions about preprojective algebras of type $A_n$.
	We consider certain projective resolutions and lifts of these resolutions to representations of $\Db(A_n)$ (Section~\ref{sec:finite:projectives}).
	We recall the notion of permutation ideals, 
        and some of their well known properties (Section~\ref{sec:finite:permutations}). 
	
	In Section~\ref{sec:category} we define a preprojective category of type $A$ over the reals (Section~\ref{sec:category:definition for R}) and over a subinterval $\II$ of $\RR$ (Section~\ref{sec:category:definition for I}), where we consider all endomorphisms ``of the same length'' to be the same.
	For $\II$, we denote this category by $\preI$.
	In the present paper we restrict our attention to open, bounded intervals $\II\subset\RR$.
	That is $\II=(a,b)$ for $a<b\in\RR$.
	In fact, we show that, up to equivalence, there are only three continuous preprojective categories of type $A$, when we consider open subintervals of $\RR$ (Proposition~\ref{prop:preprojective classes}).
	The classes are bounded (such as $(0,1)$), half bounded (such as $(0,+\infty)$), and $\RR$ itself.
	As a consequence, we consider $\II=(a,b)$ but often just use $\II=(0,1)$ for convenience.
	
	In Section~\ref{sec:category:projectives}, we discuss the representable projective $\preI$-modules.
	We also consider the lift of the representable projectives to representations of $\DbAI:=\Db(\rfp(A_{\II}))$, the bounded derived category of finitely-presented representations of a type $A_\RR$ quiver, where the quiver is rescaled to $\II$.
	The lifts are thin representations.
	
	In Section~\ref{sec:subs and quots of projs} we turn our attention to submodules and quotient modules of representable projectives $P_x$, for each $x\in\II$.
	We also consider lifts of these modules to representations of $\DbAI$, which are thin representations, and use this to guide our intuition.
	We define \emph{decorous} submodules and quotient modules 
        (Definition~\ref{def:decorous}). 
	Decorous modules include the 0 module and each representable projective.
	
	Again, we can assume $\II=(0,1)$ and we denote by $\bar{\II}$ the closure $[0,1]$.
	We show that for a short exact sequence $M\hookrightarrow P_x\twoheadrightarrow N$, we have $M$ is decorous if and only if $N$ is decorous (Proposition~\ref{prop:decorous short exact sequence}).
	We show that decorous submodules of $P_x$ are in bijection with functions $\partial:\overline{\II}\to \RR_{\geq 0}$ such that $\partial(\min\overline{\II})=x-\min\overline{\II}$, $\partial(\max\overline{\II})=\max\overline{\II}-x$, and, for all $y\neq z\in \II$, we have $|\partial(y)-\partial(z)|\leq |y-z|$ (Proposition~\ref{prop:boundaries and decorous modules}).
	The main result from this section is the following
	\begin{thm}[Corollary~\ref{cor:permuton functions and decorous submodules}]\label{thm intro:decorous}
		Let $\boldsymbol{\partial}$ be the set of all $\partial:\overline{\II}\to \RR$ such that $|\partial(y)-\partial(z)|\leq |y-z|$, for all $y,z\in\overline{\II}$, and $\partial(\max\overline{\II})\neq \partial(\min\overline{\II})\pm 1$.
		Then there is a bijection
		\[
			\xymatrix{
				\boldsymbol{\partial}/\{\partial\sim\partial' \text{ if } (\exists a\in\RR) (\forall y\in\overline{\II}) (\partial(y)=\partial'(y)+a)\}
				\ar@{<->}[d] \\
				\{\text{decorous submodules of representable projectives}\}.
			}
		\]
	\end{thm} 
	
	In Section~\ref{sec:sheets} we consider $\preI$-modules of the following form.
	Consider the sequence $M\hookrightarrow P_x \twoheadrightarrow N$ where $M$ is a decorous submodule of $P_x$ and $N$ is a decorous quotient module of $P_{x}$.
	We assume the the composition of the arrows is nonzero and define the image of $M$ in $N$ to be a \emph{sheet module}.
	We call these sheet modules because, when we lift them to representations of $\DbAI$, they are  thin representations, though a sheet module need not be indecomposable. 
	Our interest in sheet modules is an entry point into understanding general $\Hom$ spaces between $\preI$-modules.
	We define a way of understanding certain morphisms between sheet modules by analyzing the image of a generator (Proposition~\ref{prop:elementary morphisms}).
	We define such morphisms to be \emph{elementary} and describe a way of combining such morphisms as \emph{multi-elementary} (Definition~\ref{def:elementary morphisms}).
	Then we have the following conjecture.
	
	\begin{conjecture}
		A morphism between sheet modules is a finite sum of multi-elementary morphisms.
	\end{conjecture}
	
	In Section~\ref{sec:bricks} we turn our attention to bricks.
	As a consequence of Proposition~\ref{prop:deep modules are not bricks}, we see that none of following can be a brick: submodules of representable projectives, decorous quotient modules of representable projectives, and sheet modules.
	We define a \emph{sawtooth function}, denoted $\sawtooth$, and the corresponding \emph{sawtooth module} (Definitions~\ref{def:sawtooth}~and~\ref{def:sawtooth module}).
	Given a sawtooth function $\sawtooth$, we define a functor from a continuous quiver $Q$ of type $A$ to $\preI$, denoted $Z:Q\to\preI$.
	The construction has no choices and $Q$ itself is actually induced by $\sawtooth$.
	We are particularly interested in the induced functor $Z^*:\Rep(\preI)\to\Rep(Q)$.
	In Proposition~\ref{prop:bricks and decorous}, we relate bricks back to decorous submodules.
	Just before that proposition we obtain a generalization of the result in \cite{A22} just after Theorem~0.4.
	\begin{thm}[Theorem~\ref{thm:bricks are sawtooth modules from continuous A quivers}]
		Let $M$ be a $\preI$-module.
		Then $M$ is a brick if and only if $M$ is simple or $M$ is a sawtooth module.
		
		Moreover, if $M$ is a sawtooth module then the module $Z^*M$ is a brick in $\Rep(Q)$, where $Q$ is induced by the sawtooth function $\sawtooth$ that defines $M$.
	\end{thm}
	
	In Section~\ref{sec:permutons}, we begin generalizing the concept of permutation ideals. As mentioned above, the objects constituting our generalization of permutation ideals are indexed by a permuton. See Section \ref{sec:permutons} for the definition. There is an injection from permutations in $S_{n+1}$ into permutons, for any $n$; we write $\gamma_w$ for the permuton associated to the permutation $w$.

        From a permuton $\mu$, we define a collection of induced functions $(0,1)\to \RR$, one for each $x\in(0,1)$.
	Here we are using $\Lambda_{(0,1)}$ as a stand in for the equivalence class of $\preI$'s where $\II$ is open and bounded.
	In Lemma~\ref{lem:decorous from permuton}, we show that the functions induced by permutons are precisely the $\partial$ functions from Theorem~\ref{thm intro:decorous}.
	The immediate consequence is that each permuton $\mu$ defines a decorous submodule $D_\mu^a$ of $P_a$, for each $a\in (0,1)$.
	
	Just as the permutation ideals can be described as submodules of $\bigoplus_{i=1}^n P_i$, for a preprojective algebra of type $A_n$, our permuton ideals are submodules of $\bigoplus_{x\in(0,1)} P_x$ in $\Rep(\Lambda_{(0,1)})$.
	The permuton ideal is then $I_\mu =\bigoplus_{a\in\II} D_\mu^a$. 
	Consdering $\Lambda_{(0,1)}$ as a ring, we show that the ideal is two sided (Proposition~\ref{prop:permuton ideal is two sided}).
	
	In Section~\ref{recover}, we show how to recover the (picture of) the radical filtration $I^i_w$ of a permutation ideal $I_w$ as a summand of the permuton ideal corresponding to $\gamma_w$ (Theorem~\ref{finite-is-infinite}).
	We show there is a clear condition on permutons that determines when one permuton ideal contains another (Proposition~\ref{prop:nice Bruhat condition}). 
	We give a natural definition of the \emph{permuton Bruhat order}, and show that $\mu\geq \nu$ if and only if $I_\mu\subseteq I_\nu$ as modules.
	We also show that, given permutons $\gamma_u,\gamma_v$ from permutations $u,v$, respectively, we have $\gamma_u\leq \gamma_v$ in our permuton Bruhat order if and only if $u\leq v$ in the Bruhat order on permutations (Theorem~\ref{bruhat-restricts}).
	
	In Section~\ref{sec:permuaton rigidity} we extend the $\tau$-rigidity of permutation ideals to an analogous property for permuton ideals.
	For a permuton $\mu$ and $a\in (0,1)$, we have the decorous submodule $D_\mu^a$ of $P_a$.
	From Proposition~\ref{prop:decorous short exact sequence} we know that $N$, in the short exact sequence $D_\mu^a\hookrightarrow P_a \twoheadrightarrow N$, is decorous.
	We label this as $U_\mu^a$.
	In the context of preprojective algebras of type $A_n$, if $M$ is a submodule of $P_i$, we have a short exact sequence $M\hookrightarrow P_i\twoheadrightarrow \tau M$, for each projective $P_i$.
	Thus, in our context, it is natural to view $U_\mu^a$ as playing the role of $\tau D_\mu^a$, although we do not actually have an interpretation of Auslander--Reiten translation in our setting. 
	Our analogue of $\tau$-rigidity for a permuton ideal $I_\mu\subset \Lambda_{(0,1)}$ is the following.
	\begin{thm}[Theorem~\ref{cont-rigid}]
		Let $\mu$ be a permuton and $a,b\in(0,1)$.
		Then we have $\Hom(D_\mu^a,U_\mu^b)=0$.
	\end{thm}
	
	\subsection{Funding}
	JDR is supported by FWO grant 1298325N and was also partially supported by the FWO grants G0F5921N (Odysseus) and G023721N, and by the KU Leuven grant iBOF/23/064.
		HT is supported by  NSERC Discovery Grant RGPIN-2022-03960 and the Canada Research Chairs program, grant number CRC-2021-00120. 
	
	\section{Recollections on the preprojective algebra of type $A_n$}\label{sec:finite}
	In this section we recall some known facts and properties about preprojective algebras of type $A_n$.
	In Section~\ref{sec:finite:projectives}, we recall some useful ways of interpreting indecomposable projective representations and some projective resolutions.
	In Section~\ref{sec:finite:permutations}, we recall the link between certain ideals of a preprojective algebra of type $A_{n-1}$ and the Weyl group $\mathfrak{S}_{n}$.
	Then we relate this link to the Auslander--Reiten translation of a submodule of a projective indecomposable.

	\subsection{Projective resolutions}\label{sec:finite:projectives}
	Here we will consider preprojective algebras of type $A_n$ and some projective resolutions in a particular formulation that will be useful in Section~\ref{sec:category}.
	Specifically, the goal of this section is to help the reader understand the following perspective, which we state in the form of a proposition.
	
	\begin{proposition}\label{prop:finite lifting statement}
		Let $\Lambda$ be the preprojective (path) algebra of type $A_n$, for $n\geq 3$, and let $M$ be a representation of $\Lambda$ with projective presentation
		\[
			\xymatrix{
				\cdots \ar[r] & P_j \ar[r]^-d & P_i \ar[r] & M \ar[r] & 0,
			}
		\]
		where $P_i$ and $P_j$ are the indecomposable projectives at $i$ and $j$, respectively.
		Then we may lift these last two nonzero morphisms to the following exact diagram in $\repp(\mathcal{D}^b(A_n))$:
		\[
			\xymatrix{
				P_{b,j} \ar[r]^-{\tilde{d}} & P_{a,i} \ar[r] & \widetilde{M} \ar[r] & 0,
			}
		\]
		where $P_{a,i}$ and $P_{b,j}$ are $\Hom_{\repp(\mathcal{D}^b(A_5))}((a,i),-)$ and $\Hom_{\repp(\mathcal{D}^b(A_5))}((b,j),-)$, respectively.
		Moreover, $P_{a,i}$, $P_{b,j}$, $\tilde{d}$, and $\widetilde{M}$ are lifts of $P_i$, $P_j$, $d$, and $M$, respectively.
	\end{proposition}
	The indexing for the lifts can be seen in Figure~\ref{fig:bounded derived indexing} and the lift $\tilde{d}$ of the morphism $d$ is visualized in Figure~\ref{fig:finite lift}.
	\medskip
		
	Consider the Dynkin diagram $A_5$ and let $\Lambda$ be the preprojective algebra of type $A_5$.
	We have the quiver
	\[
		\xymatrix{
			1  \ar@<0.5ex>[r]^-{\alpha_1} &
			2  \ar@<0.5ex>[l]^-{\alpha_1^*} \ar@<0.5ex>[r]^-{\alpha_2} &
			3  \ar@<0.5ex>[l]^-{\alpha_2^*} \ar@<0.5ex>[r]^-{\alpha_3} &
			4  \ar@<0.5ex>[l]^-{\alpha_3^*} \ar@<0.5ex>[r]^-{\alpha_4} &
			5, \ar@<0.5ex>[l]^-{\alpha_4^*}
		}
	\]
	where $\alpha_i^*\alpha_i = \alpha_{i-1}\alpha_{i-1}^*$ for all $i\in\mathbb{Z}$ and we consider $\alpha_i=\alpha_i^*=0$ for $i\in\mathbb{Z}\setminus\{1,2,3,4\}$.
	\medskip
	
	We can think of $\Lambda$ as a category in the following way, which is in the same form as the description of our new categories in Section~\ref{sec:category}.
	Let $Q$ be the quiver above without the relation on $\alpha$'s and $\alpha^*$'s.
	Let $\mathcal{Q}$ be the category whose objects are $\{1,2,3,4,5\}$ and whose morphisms are given by $\Hom_{\mathcal{Q}}(i,j) = e_j(\Bbbk Q) e_i$.
	Then $\Bbbk$-linear functors $\mathcal{Q}$ to $\Bbbk$-vector spaces are exactly representations of $Q$ over $\Bbbk$.
	We then let $I$ be the ideal in $\mathcal{Q}$ generated by the relations $\alpha^*_i\alpha_i = \alpha_{i-1}\alpha^*_{i-1}$ as before, where we have $\alpha_i=\alpha^*_i=0$ if $i\notin\{1,2,3,4\}$, for the purposes of notation.
	We can think of $\mathcal{Q}$ as being the same as $\Bbbk Q$.
	Similarly, we think of $\mathcal{Q}/I$ as being the same as $\Lambda$.
	Overloading notation, we write $\Lambda$ for $\mathcal{Q}/I$.
%
	The dimension of $\Hom_{\Lambda}(i,j)$ can be seen in the following table:
	\begin{center}
	\begin{tabular}{l|c|c|c|c|l}
		${\downarrow}i\ j{\rightarrow}$ & 1 & 2 & 3 & 4 & 5 \\ \hline
		1 & 1 & 1 & 1 & 1 & 1 \\ \hline
		2 & 1 & 2 & 2 & 2 & 1 \\ \hline
		3 & 1 & 2 & 3 & 2 & 1 \\ \hline
		4 & 1 & 2 & 2 & 2 & 1 \\ \hline
		5 & 1 & 1 & 1 & 1 & 1. 
	\end{tabular}
	\end{center}

	\begin{definition}\label{def:path like}
		A composition of $\alpha$'s and $\alpha^*$'s is called \emph{path like} if it is nonzero. The \emph{length} of a pathlike morphism is the number of $\alpha$'s and $\alpha^*$'s appearing in it. 
	\end{definition}

Observe that $\Hom_\Lambda(i,j)$ has a canonical basis consisting of pathlike morphisms. 
        Consider, for example, $\Hom_{\Lambda}(2,2)$.
	This has a canonical basis of $\{\boldsymbol{1}_2,f\}$, where $f=\alpha^*_2\alpha_2=\alpha_1\alpha^*_1$. The morphism $\boldsymbol{1}_2$ has length 0 and the morphism $\alpha^*_2\alpha_2$ as has length 2.

        There is at most one morphism from $i$ to $j$ of length $k$, for any non-negative integer $k$. Thus, we describe $\Hom(i,j)$ as $\Bbbk^{\{\ell_1,\dots,\ell_r\}}$, where $\ell_1,\dots,\ell_r$ are the lengths of the pathlike morphisms from $i$ to $j$. 
	For example, $\Hom_{\Lambda}(2,2)\cong \Bbbk^{\{0,2\}}$.
	In particular, $\End_{\Lambda}(2)=\Hom_{\Lambda}(2,2)$ is a graded ring.
	
	Similarly, if we consider $\Hom_{\Lambda}(2,4)$, this has basis $\{\alpha_3\alpha_2,f\}$ where $f$ is the equivalence class of $\alpha^*_4\alpha_4\alpha_3\alpha_2$.
	In this case the respective lengths of the two basis elements are 2 and 4.
	So, $\Hom_{\Lambda}(2,4)\cong\Bbbk^{\{2,4\}}$.

	We consider our projective $P_i$ to be the functor $\Hom_{\Lambda}(i,-)$, which is precisely the projective representation at $i$.
	Then, $P_i(j)=\Hom_{\Lambda}(i,j)$.
	
	We consider a morphism $d:P_2\to P_3$ defined as follows.
	The morphism $d$ is entirely determined by the morphism $d(2):\Hom_{\Lambda}(2,2)\to\Hom_{\Lambda}(3,2)$ since $\boldsymbol{1}_2$ is the generator of $P_2$. 
	Using our descriptions above, we consider $P_3(2)=\Hom_{\Lambda}(3,2)\cong\Bbbk^{\{1,3\}}$, where the morphism in $\Hom_{\Lambda}(3,2)$ of length $3$ is the equivalence class of $\alpha^*_2 \alpha_2 \alpha^*_2$.
	Recall $\boldsymbol{1}_2\in P_2(2)=\End_{\Lambda}(2,2)$.
	Define $d(\boldsymbol{1}_2)=\alpha^*_2\alpha_2\alpha^*_2$.
	
	Now we consider the module $M$ whose projective presentation is
	\[
		\xymatrix{
			{\cdots} \ar[r] & P_2 \ar[r]^-{d} & P_3 \ar@{->>}[r] & M.
		}
	\]
%
	Then we can see $M$ is the following representation:
	\[
		\xymatrix@C=12ex{
			\Bbbk^{\{2\}} \ar@<1ex>[r]^-{\left[\begin{array}{c}0\end{array}\right]} &
			\Bbbk^{\{1\}} \ar@<1ex>[r]^-{\left[\begin{array}{cc}0 & 1\end{array}\right]} \ar@<1ex>[l]^-{\left[\begin{array}{c}1\end{array}\right]} &
			\Bbbk^{\{0,2\}} \ar@<1ex>[r]^-{\left[\begin{array}{cc} 1 & 0 \\ 0 & 1\end{array}\right]} \ar@<1ex>[l]^-{\left[\begin{array}{c} 1 \\ 0\end{array}\right]}&
			\Bbbk^{\{1,3\}} \ar@<1ex>[r]^-{\left[\begin{array}{c} 1 \\ 0 \end{array}\right]} \ar@<1ex>[l]^-{\left[\begin{array}{cc}0 & 0 \\ 1 & 0\end{array}\right]} &
			\Bbbk^{\{2\}}. \ar@<1ex>[l]^-{\left[\begin{array}{cc} 0 & 1 \end{array}\right]}
		}
	\]
	
	We can ``lift'' the morphism $d$ to $\repp(\mathcal{D}^b(A_5))$ in the following way.
	First, we choose a way of identifying vertices in the Auslander-Reiten quiver of $\mathcal{D}^b(A_5)$, which can be seen in Figure~\ref{fig:bounded derived indexing}.
	\begin{figure}
	\begin{center}
		\begin{displaymath}
		\begin{tikzpicture}
			\foreach \x in {0,1,2,3,4,5,6,7,8,9}
			{
				\foreach \y in {0,1,2}
				{
					\filldraw[fill=black] (\x,\y) circle[radius=.3mm];
				}
			}
			\foreach \x in {0,1,2,3,4,5,6,7,8}
			{
				\foreach \y in {0,1}
				{
					\draw[->] (\x + 0.1 , \y + 0.1) -- (\x + 0.4 , \y + 0.4);
				}
				\foreach \y in {1,2}
				{
					\draw[->] (\x + 0.1 , \y - 0.1) -- (\x + 0.4 , \y - 0.4);
				}
			}
			\foreach \x in {0.5,1.5,2.5,3.5,4.5,5.5,6.5,7.5,8.5}
			{
				\foreach \y in {0.5,1.5}
				{
					\filldraw[fill=black] (\x,\y) circle[radius=.3mm];
					\draw[->] (\x + 0.1 , \y + 0.1) -- (\x + 0.4 , \y + 0.4);
					\draw[->] (\x + 0.1 , \y - 0.1) -- (\x + 0.4 , \y - 0.4);
				}
			}
			\foreach \x in {-0.5,9.5}
			{
				\foreach \y in {0.5,1.5}
				{
					\draw (\x,\y) circle[radius=.3mm];
				}
			}
			\foreach \y in {0,1}
			{
				\draw[->, dotted] (9.1 , \y + 0.1) -- (9.4 , \y + 0.4);
			}
			\foreach \y in {1,2}
			{
				\draw[->, dotted] (9.1 , \y - 0.1) -- (9.4 , \y - 0.4);
			}
			\foreach \y in {0.5,1.5}
			{
				\draw[->, dotted] (-0.4 , \y + 0.1) -- (-0.1 , \y + 0.4);
				\draw[->, dotted] (-0.4 , \y - 0.1) -- (-0.1 , \y - 0.4);
			}
			\filldraw[fill opacity=.3, fill=red, draw opacity =0] (4.5,1.5) circle[radius=1.5mm];
			\filldraw[fill opacity=.3, fill=blue, draw opacity =0] (6,2) circle[radius=1.5mm];
			\draw (-0.5,-0.25) node[anchor=north] {$\cdots$};
			\draw (0,0) node[anchor=north] {-1};
			\draw (1,0) node[anchor=north] {1};
			\draw (2,0) node[anchor=north] {3};
			\draw (3,0) node[anchor=north] {5};
			\draw (4,0) node[anchor=north] {7};
			\draw (5,0) node[anchor=north] {9};	
			\draw (6,0) node[anchor=north] {11};	
			\draw (7,0) node[anchor=north] {13};	
			\draw (8,0) node[anchor=north] {15};	
			\draw (9,0) node[anchor=north] {17};	
			\draw (9.5,-0.25) node[anchor=north] {$\cdots$};
			\draw (0.5,-0.5) node[anchor=north] {0};
			\draw (1.5,-0.5) node[anchor=north] {2};
			\draw (2.5,-0.5) node[anchor=north] {4};
			\draw (3.5,-0.5) node[anchor=north] {6};
			\draw (4.5,-0.5) node[anchor=north] {8};
			\draw (5.5,-0.5) node[anchor=north] {10};
			\draw (6.5,-0.5) node[anchor=north] {12};
			\draw (7.5,-0.5) node[anchor=north] {14};
			\draw (8.5,-0.5) node[anchor=north] {16};
			\draw (-1,0) node[anchor=east] {1};
			\draw (-1,1) node[anchor=east] {3};
			\draw (-1,2) node[anchor=east] {5};
			\draw (-1.5,0.5) node[anchor=east] {2};
			\draw (-1.5,1.5) node[anchor=east] {4};
		\end{tikzpicture}
		\end{displaymath}
		\caption{In the diagram, the vertex highlighted in \textcolor{red}{red} has coordinates $(8,4	)$.
		The vertex highlighted in \textcolor{blue}{blue} has coordinates $(11,5)$.}\label{fig:bounded derived indexing}
	\end{center}
	\end{figure}
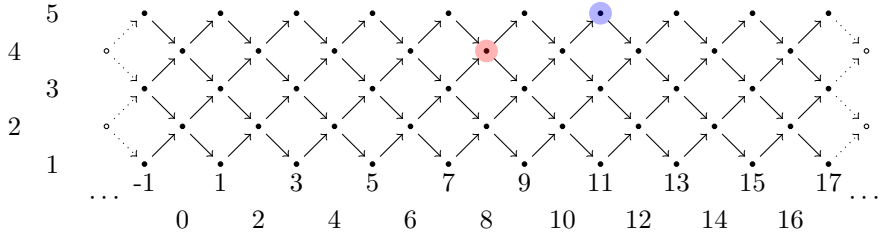
	
	Let $Q$ be the quiver $1\leftarrow 2\leftarrow 3\leftarrow 4\leftarrow 5$.
	One way to assign the numbering is to assign the coordinates $(i,i)$ to the vertex associated to $P_i[0]$ in $\mathcal{D}^b(Q)= \mathcal{D}^b(A_5)$.
	From here we use the fact that the Auslander-Reiten quiver of $\mathcal{D}^b(A_5)$ has vertices $\ZZ Q_0$.
	Assign $(i-2,j)$ to the vertex $\tau(i,j)$ and assign $(i+2,j)$ to the vertex $\tau^{-1}(i,j)$, where $\tau$ is the Auslander-Reiten translation. 

	We want to ``lift'' the sequence $P_2\to P_3\twoheadrightarrow M$ to a sequence of morphisms in $\mathcal{D}^b(A_5)$.
	This perspective will guide our intuition in Sections \ref{sec:category}, \ref{sec:subs and quots of projs}, \ref{sec:sheets}, and \ref{sec:bricks}.
	Since $P_3$ is the projective cover of $M$ in $\repp(\Lambda)$, we ``lift'' it to $P_{i,3}$ in $\repp(\mathcal{D}^b(A_5))$, for some odd $i\in\mathbb{Z}$.
	To make things easy, we pick $i=1$.
	Since $d(\boldsymbol{1}_2)=\alpha^*_2\alpha_2\alpha^*_2$ has length 3, we ``lift'' $P_2$ to $P_{i+3,2}$, which is $P_{4,2}$.
	Then we look at the unique nonzero map $\tilde{d}:\textcolor{blue}{P_{4,2}}\to \textcolor{red}{P_{1,3}}$ in $\repp(\mathcal{D}^b(A_5))$.
	The lifting can be seen in Figure~\ref{fig:finite lift}
	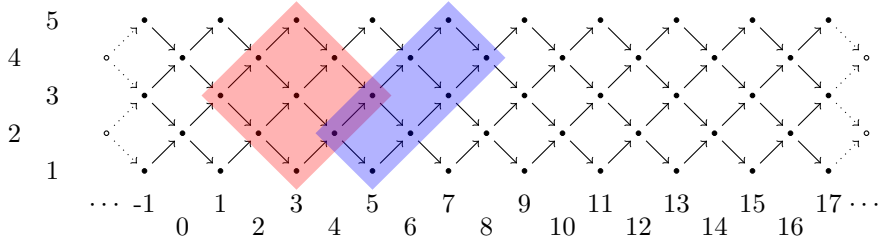
\begin{figure}
	\begin{center}
		\begin{displaymath}
		\begin{tikzpicture}
			\foreach \x in {0,1,2,3,4,5,6,7,8,9}
			{
				\foreach \y in {0,1,2}
				{
					\filldraw[fill=black] (\x,\y) circle[radius=.3mm];
				}
			}
			\foreach \x in {0,1,2,3,4,5,6,7,8}
			{
				\foreach \y in {0,1}
				{
					\draw[->] (\x + 0.1 , \y + 0.1) -- (\x + 0.4 , \y + 0.4);
				}
				\foreach \y in {1,2}
				{
					\draw[->] (\x + 0.1 , \y - 0.1) -- (\x + 0.4 , \y - 0.4);
				}
			}
			\foreach \x in {0.5,1.5,2.5,3.5,4.5,5.5,6.5,7.5,8.5}
			{
				\foreach \y in {0.5,1.5}
				{
					\filldraw[fill=black] (\x,\y) circle[radius=.3mm];
					\draw[->] (\x + 0.1 , \y + 0.1) -- (\x + 0.4 , \y + 0.4);
					\draw[->] (\x + 0.1 , \y - 0.1) -- (\x + 0.4 , \y - 0.4);
				}
			}
			\foreach \x in {-0.5,9.5}
			{
				\foreach \y in {0.5,1.5}
				{
					\draw (\x,\y) circle[radius=.3mm];
				}
			}
			\foreach \y in {0,1}
			{
				\draw[->, dotted] (9.1 , \y + 0.1) -- (9.4 , \y + 0.4);
			}
			\foreach \y in {1,2}
			{
				\draw[->, dotted] (9.1 , \y - 0.1) -- (9.4 , \y - 0.4);
			}
			\foreach \y in {0.5,1.5}
			{
				\draw[->, dotted] (-0.4 , \y + 0.1) -- (-0.1 , \y + 0.4);
				\draw[->, dotted] (-0.4 , \y - 0.1) -- (-0.1 , \y - 0.4);
			}
			\filldraw[fill opacity = .3, draw opacity = 0, fill = red] (.75,1) -- (2,2.25) -- (3.25,1) -- (2,-0.25) -- cycle;
			\filldraw[fill opacity = .3, draw opacity = 0, fill = blue] (2.25,0.5) -- (3,-0.25) -- (4.75,1.5) -- (4,2.25) -- cycle;
			\draw (-0.5,-0.25) node[anchor=north] {$\cdots$};
			\draw (0,-.2) node[anchor=north] {-1};
			\draw (1,-.2) node[anchor=north] {1};
			\draw (2,-.2) node[anchor=north] {3};
			\draw (3,-.2) node[anchor=north] {5};
			\draw (4,-.2) node[anchor=north] {7};
			\draw (5,-.2) node[anchor=north] {9};	
			\draw (6,-.2) node[anchor=north] {11};	
			\draw (7,-.2) node[anchor=north] {13};	
			\draw (8,-.2) node[anchor=north] {15};	
			\draw (9,-.2) node[anchor=north] {17};	
			\draw (9.5,-0.25) node[anchor=north] {$\cdots$};
			\draw (0.5,-0.5) node[anchor=north] {0};
			\draw (1.5,-0.5) node[anchor=north] {2};
			\draw (2.5,-0.5) node[anchor=north] {4};
			\draw (3.5,-0.5) node[anchor=north] {6};
			\draw (4.5,-0.5) node[anchor=north] {8};
			\draw (5.5,-0.5) node[anchor=north] {10};
			\draw (6.5,-0.5) node[anchor=north] {12};
			\draw (7.5,-0.5) node[anchor=north] {14};
			\draw (8.5,-0.5) node[anchor=north] {16};
			\draw (-1,0) node[anchor=east] {1};
			\draw (-1,1) node[anchor=east] {3};
			\draw (-1,2) node[anchor=east] {5};
			\draw (-1.5,0.5) node[anchor=east] {2};
			\draw (-1.5,1.5) node[anchor=east] {4};whose
		\end{tikzpicture}
		\end{displaymath}
		\caption{In red, the support of $P_{1,3}$ and, in blue, the support of $P_{4,2}$. These are lifts of $P_3$ and $P_2$, respectively, such that the image of the morphism of $\mathcal{D}^b(A_5)$-modules $P_{4,2}\to P_{1,3}$ coincides with a lift of $M\cong P_3/\mathrm{im}(d)$.}\label{fig:finite lift}
	\end{center}
	\end{figure}
	
	Set $\widetilde{M}:=\coker(\tilde{d})$.
	The support of $\widetilde{M}$ is given by the vertices in the red region that are not also in the blue region.
	That is, the support of $\widetilde{M}$ is the support of $P_{1,3}$ without $\{(4,2),(5,3)\}$.
	The representation $\widetilde{M}$ can be ``pushed down'' to $M$ by the following assignment.
	We write $\Bbbk_{i,j}$ to mean the copy of $\Bbbk$ at $(i,j)$ in the support of $\widetilde{M}$.
	\begin{align*}
		1\in\Bbbk_{3,1} &\mapsto 1 \in M(1) &
		1\in\Bbbk_{2,2} &\mapsto 1 \in M(2) \\
		1\in\Bbbk_{1,3} &\mapsto (1,0)\in M(3) &
		1\in\Bbbk_{3,3} &\mapsto (0,1)\in M(3) \\
		1\in\Bbbk_{2,4} &\mapsto (1,0)\in M(4) &
		1\in\Bbbk_{4,4} &\mapsto (0,1)\in M(4) \\
		1\in\Bbbk_{3,5} &\mapsto 1\in M(5)
	\end{align*}	
	
	Nothing about this procedure used the ``5-ness'' of the 5 in $A_5$.
	In particular, this technique works for such resolutions for any module $M$ of a preprojective algebra $\Lambda_n$ of type $A_n$ whose projective presentation is of the form
	\[
		\xymatrix{
			{\cdots} \ar[r] & P_j \ar[r]^-d & P_i \ar@{->>}[r] & M,
		}
	\]
	when $d(\boldsymbol{1}_j)$ is (a scalar multiple of) a path in $P_i(j)=\Hom_{\Lambda_n}(i,j)$.
	We will later generalize this technique to our continuous preprojective categories in Sections~\ref{sec:subs and quots of projs}~and~\ref{sec:sheets}.

	\subsection{Permutation ideals}\label{permutation-ideals}\label{sec:finite:permutations}
	
	There is a class of ideals of a preprojective algebra which is in bijection with the elements of the corresponding Weyl group. See \cite{Miz} for the simply-laced Dynkin case, which is the case relevant here. We will describe this correspondence in the case that $\Lambda$ is the preprojective algebra of type $A_{n-1}$. In this case, the corresponding Weyl group is $\S_n$, the group of permutations of $\{1,\dots,n\}$. We call the ideals corresponding to permutations, \emph{permutation ideals}.
	In Section~\ref{sec:permutons}, we will define an analogous collection of ideals in the continuous case. 
	
	For $1\leq i \leq n-1$, consider a non-zero map from $\Lambda$ to $S_i$, and let $I_i$ be its kernel. $I_i$ is a two-sided ideal.
		Also, for $1\leq i\leq n-1$, let $s_i=(i\ i+1)\in \S_n$ be the transposition of $i$ and $i+1$. We refer to the elements $s_i$ as \emph{adjacent transpositions}.
	For $w\in \S_n$, a \emph{reduced expression} for $w$ is an expression for $w$ as a product of adjacent transpositions of minimum possible length. We write $\ell(w)$ for this minimum possible length.

	\begin{theorem}[{\cite[Theorem 2.14]{Miz}}]\label{mizuno-thm} There is an injection from $W$ to two-sided ideals of $\Lambda$, sending $w$ to $I_w$, where $I_w$ is defined by taking a reduced expression for $w$, say $w=s_{i_1}\dots s_{i_r}$, and letting $I_w=I_{i_1}\dots I_{i_r}$. \end{theorem}
	
	In particular, the above definition of $I_w$ does not depend on the choice of reduced expression for $w$. We now give a more algorithmic approach to describing $I_w$. 
	
	\begin{lemma}\label{algo} Suppose $\ell(ws_i)=\ell(w)+1$.
	Then  $I_wI_{s_i}$ is obtained from $I_w$ by removing the copies of $S_i$ in the top of $I_w$. \end{lemma}
	
	\begin{proof}
	  Consider the short exact sequence $0\rightarrow I_{i} \rightarrow \Lambda \rightarrow S_i \rightarrow 0$, and tensor with $I_w$, obtaining the following exact sequence:
	  $$ I_{w}\otimes_\Lambda I_{i} \rightarrow I_w \otimes_\Lambda \Lambda \rightarrow I_w\otimes_\Lambda S_i \rightarrow 0$$
	Now  $I_w\otimes_\Lambda S_i$ consists of the copies of $S_i$ in the top of $I_w$. The image of $I_w\otimes_\Lambda I_{i}$ inside $I_w\otimes \Lambda \simeq I_w \subset \Lambda$ is $I_wI_i=I_{ws_i}$. Then, exactness at $I_w\otimes_\Lambda \Lambda$ proves the lemma.
	\end{proof}
	
	We write $P_i$ for the indecomposable projective associated to vertex $i$. We write $(I_w)^i$ for the summand of $I_w$ contained in $P_i$.

	We will now look at some examples of the radical filtration of $I_w$. For reasons that will become evident later, we want to be more specific than usual about exactly how we are drawing the radical filtrations. We coordinatize the plane of the page so that $x$ records the horizontal position, increasing from left to right, and $y$ records the vertical position, increasing downwards.  Each simple $S_i$ is drawn in a square shape (rotated 45 degrees with respect to the $(x,y)$ coordinates), so that each square for $S_i$ is centered at $x=i/n$. (For a technical reason which will become clear in Section \ref{recover}, we also remove the lower boundary of each square.) This has the effect that each summand is drawn between $x=0$ and $x=1$. For greater legibility, we do not draw them superimposed.
	
	\begin{example} \label{radical-example} Let us consider some examples for the preprojective algebra of type $A_4$. \begin{enumerate} \item
	Here is the radical filtration of the algebra considered as a right module over itself. It has four indecomposable summands.

	$$
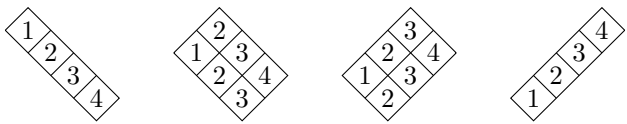
\begin{tikzpicture}[scale=1.5]
	  \draw (0,4/5)--(1/5,1)--(1,1/5)--(4/5,0)--(0,4/5);
	  \draw (1/5,3/5)--(2/5,4/5);
	  \draw (2/5,2/5)--(3/5,3/5);
	  \draw (3/5,1/5)--(4/5,2/5);
	  \node (a) at (1/5,4/5) {1};
	  \node (b) at (2/5,3/5) {2};
	  \node (c) at (3/5,2/5) {3};
	  \node (d) at (4/5,1/5) {4};
	\end{tikzpicture}
	\qquad
	\begin{tikzpicture}[scale=1.5]
	  \draw (0,3/5)--(2/5,1)--(1,2/5)--(3/5,0)--(0,3/5);
	  \draw (1/5,4/5)--(4/5,1/5);
	  \draw (3/5,4/5)--(1/5,2/5);
	  \draw (4/5,3/5)--(2/5,1/5);
	  \node (a) at (1/5,3/5) {1};
	  \node (b) at (2/5,4/5) {2};
	  \node (c) at (2/5,2/5) {2};
	  \node (d) at (3/5,1/5) {3};
	  \node (e) at (3/5,3/5) {3};
	  \node (f) at (4/5,2/5) {4};
	\end{tikzpicture}
	\qquad
	\begin{tikzpicture}[xscale=-1.5,yscale=1.5]
	  \draw (0,3/5)--(2/5,1)--(1,2/5)--(3/5,0)--(0,3/5);
	  \draw (1/5,4/5)--(4/5,1/5);
	  \draw (3/5,4/5)--(1/5,2/5);
	  \draw (4/5,3/5)--(2/5,1/5);
	  \node (a) at (1/5,3/5) {4};
	  \node (b) at (2/5,4/5) {3};
	  \node (c) at (2/5,2/5) {3};
	  \node (d) at (3/5,1/5) {2};
	  \node (e) at (3/5,3/5) {2};
	  \node (f) at (4/5,2/5) {1};
	\end{tikzpicture}
	\qquad
	\begin{tikzpicture}[xscale=-1.5,yscale=1.5]
	  \draw (0,4/5)--(1/5,1)--(1,1/5)--(4/5,0)--(0,4/5);
	  \draw (1/5,3/5)--(2/5,4/5);
	  \draw (2/5,2/5)--(3/5,3/5);
	  \draw (3/5,1/5)--(4/5,2/5);
	  \node (a) at (1/5,4/5) {4};
	  \node (b) at (2/5,3/5) {3};
	  \node (c) at (3/5,2/5) {2};
	  \node (d) at (4/5,1/5) {1};
	\end{tikzpicture}$$
	
	\item
	The radical filtration of $I_i$ is the same as that of $\Lambda$, except that it is missing the $S_i$ in the top of $P_i$.
	
	\item 
	Let us now consider a more generic example, such as $w=25341$. One possible reduced expression for $w$ is $(12)(23)(45)(34)(23)(45)$. Using this reduced expression, and applying the algorithm in Lemma \ref{algo}, we successively remove, from the top, all copies of $S_1, S_2, S_4, S_3, S_2, S_4$. A different reduced expression for $w$, such at $(12)(23)(34)(45)(34)(23)$, would remove the same simples in total, but in a different order. 
	The resulting radical filtration of the remaining ideal is show below (with the entire $\Lambda_\Lambda$ still visible in grey).
	
	$$
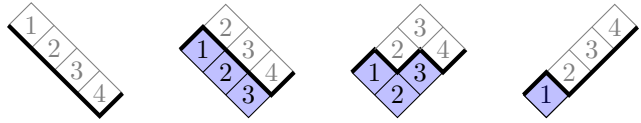
\begin{tikzpicture}[scale=1.5]
	  
	  \draw[gray] (0,4/5)--(1/5,1)--(1,1/5)--(4/5,0)--(0,4/5);
	  \draw[gray] (1/5,3/5)--(2/5,4/5);
	  \draw[gray] (2/5,2/5)--(3/5,3/5);
	  \draw[gray] (3/5,1/5)--(4/5,2/5);
	  \node[gray] (a) at (1/5,4/5) {1};
	  \node[gray] (b) at (2/5,3/5) {2};
	  \node[gray] (c) at (3/5,2/5) {3};
	  \node[gray] (d) at (4/5,1/5) {4};
	  \draw[ultra thick] (0,4/5)--(4/5,0)--(1,1/5);
	\end{tikzpicture}
	\qquad
	\begin{tikzpicture}[scale=1.5]
	  \draw[fill=blue, fill opacity=.25] (4/5,1/5)--(3/5,0)--(0,3/5)--(1/5,4/5)--(4/5,1/5);
	  \draw[gray] (1/5,4/5) -- (2/5,1)-- (1,2/5)--(4/5,1/5); 
	  \draw[gray] (3/5,4/5)--(2/5,3/5);
	  \draw (2/5,3/5)--(1/5,2/5);
	  \draw[gray] (4/5,3/5)--(3/5,2/5);
	  \draw (3/5,2/5)--(2/5,1/5);
	  \node (a) at (1/5,3/5) {1};
	  \node[gray] (b) at (2/5,4/5) {2};
	  \node (c) at (2/5,2/5) {2};
	  \node (d) at (3/5,1/5) {3};
	  \node[gray] (e) at (3/5,3/5) {3};
	  \node[gray] (f) at (4/5,2/5) {4};
	  \draw[ultra thick] (0,3/5)--(1/5,4/5)--(4/5,1/5)--(1,2/5);
	\end{tikzpicture}
	\qquad
	\begin{tikzpicture}[xscale=1.5,yscale=1.5]
	  \draw [fill=blue, fill opacity=.25] (0,2/5)--(1/5,3/5) -- (2/5,2/5)--(3/5,3/5)--(4/5,2/5)--(2/5,0)--(0,2/5);
	  \draw (1/5,1/5)--(2/5,2/5)--(3/5,1/5);
	  \draw [ultra thick] (0,2/5)--(1/5,3/5) -- (2/5,2/5)--(3/5,3/5)--(4/5,2/5)--(1,3/5);
	  \draw[gray] (1/5,3/5)--(3/5,1)--(1,3/5);
	  \draw [gray] (2/5,4/5)--(3/5,3/5)--(4/5,4/5);
	  \node[gray] (a) at (4/5,3/5) {4};
	  \node[gray] (b) at (3/5,4/5) {3};
	  \node (c) at (3/5,2/5) {3};
	  \node (d) at (2/5,1/5) {2};
	  \node[gray] (e) at (2/5,3/5) {2};
	  \node (f) at (1/5,2/5) {1};
	\end{tikzpicture}
	\qquad
	\begin{tikzpicture}[xscale=-1.5,yscale=1.5]
	  \draw[gray] (0,4/5)--(1/5,1)--(1,1/5)--(4/5,0)--(0,4/5);
	  \draw[gray] (1/5,3/5)--(2/5,4/5);
	  \draw[gray] (2/5,2/5)--(3/5,3/5);
	  \draw[gray] (3/5,1/5)--(4/5,2/5);
	  \node[gray] (a) at (1/5,4/5) {4};
	  \node[gray] (b) at (2/5,3/5) {3};
	  \node[gray] (c) at (3/5,2/5) {2};
	  \node (d) at (4/5,1/5) {1};
	  \draw[fill=blue, fill opacity=.25] (1,1/5)--(4/5,2/5)--(3/5,1/5)--(4/5,0)--(1,1/5);
	  \draw [ultra thick] (0,4/5)--(3/5,1/5)--(4/5,2/5)--(1,1/5);
	\end{tikzpicture}
	$$
	In this example, $(I_w)^1$ is actually zero, because all four composition factors of $P_1$ are removed. 
	\end{enumerate} \end{example}
	
	Notice that $(I_w)^i$ can be encoded by the piecewise linear curve that separates the composition factors of $P_i$ which are in the ideal from those which are not. We always draw this curve so that it stretches all the way from $x=0$ to $x=1$, passing immediately under the bottommost composition factors of $P_i$ which are not in the $(I_w)^i$. So as to make this curve uniquely defined as a function, we fix our coordinates so that the indecomposable projectives have the top vertex of the corresponding rectangle at $y=0$. It follows that the bottom of the rectangle is at $y=1$.
	
	\begin{example} Continuing Example \ref{radical-example}(3), 
	the function defining $(I_w)^1$ is
	$f_1(x)=1-|x-4/5|$. The function defining $(I_w)^2$ is $2/5-x$ for $0\leq x\leq 1/5$, it is $x$ for $1/5\leq x\leq 4/5$, and $8/5-x$ for $4/5\leq x\leq 1$.  
	\end{example}
	
	\subsection{Auslander--Reiten translation} 
	
	The Auslander--Reiten translation of a module $M$, generally denoted $\tau M$,
	is calculated as follows. Take a minimal projective presentation of $M$, say
	$$Q\xrightarrow{g} P \rightarrow M \rightarrow 0.$$
	Then $\tau M$ is by definition the kernel of $\nu g$, where $\nu$ is the Nakayama functor sending $P_i$ to the corresponding injective. See \cite[Chapter IV]{ASS} for more details.
	
	We will need the following lemma:
	
	\begin{lemma}\label{discrete-tau} Let $M$ be a submodule of the indecomposable projective module $P_i$. Then $\tau M \simeq P_i/M$. \end{lemma}
	
	\begin{proof} If we draw the piecewise linear curve separating the composition factors in $M$ from the composition factors not in $M$, and going from the left corner of the projective to the right corner of the projective (as we have already done to describe $(I_w)^i$), then the terms in the minimal projective presentation can be described as follows. If the function has a local minimum at $x=i/n$, there is a summand $P_i$ in $P$. If the function has a local maximum at $x=i/n$, then there is a summand $P_i$ in $Q$. 
	When we apply $\nu$ to $Q\xrightarrow{g} P$, it changes from a projective presentation of $M$ to an injective copresentation of $P_i/M$, which proves the lemma. \end{proof}
	  
	A module is called $\tau$-rigid if $\Hom(M,\tau M)=0$. Mizuno showed the following result (in greater generality):
	
	\begin{theorem}[{\cite[Proposition 2.10]{Miz}}]\label{mizuno-rigid} $I_w$ is $\tau$-rigid. \end{theorem}
	
	\section{A preprojective category}\label{sec:category}
	In this section we seek to define preprojective categories of type $A$ over $\RR$ and over an interval $\II\subset\RR$, and to describe the representable projectives in the case when $\II\subset\RR$ is of the form $(a,b)$, for $a<b\in\RR$.
	
	\subsection{Definition for $\RR$}\label{sec:category:definition for R}
	In this section we define a preprojective category over the real line.
	
	Let $\RR$ have the usual total order.
	We construct a category $\preR$ in the following way.
	We set $\mathrm{Ob}(\preR)=\RR$.
	For each $x<y$ we have the morphisms $\alpha_{yx}:x\to y$ and $\alpha^*_{xy}:y\to x$.
	(The $\alpha$'s go right and the $\alpha^*$'s go left.)
	For notational purposes, we identify the morphisms $\alpha_{xx}=\alpha^*_{xx}=\boldsymbol{1}_x$.
	The $\alpha$'s and $\alpha^*$'s are subject to the following relations:
	\begin{enumerate}
		\item\label{def:relation:pause} $\alpha_{zy}\alpha_{yx} = \alpha_{zx}$ and $\alpha^*_{xy}\alpha^*_{yz}=\alpha^*_{xz}$, for $x\leq y\leq z$ in $\RR$, and
		\item\label{def:relation:equivalent} $\alpha^*_{x+a-b,x+a}\alpha^{}_{x+a,x}=\alpha^{}_{x+a-b,x-b}\alpha^*_{x-b,x}$ for $a,b$ positive reals.
	\end{enumerate}
	When we write $\alpha^{(*)}_{xy}$ we mean whichever of $\alpha_{xy}$ or $\alpha^*_{xy}$ exists, depending on whether $y\leq x$ or $x\leq y$, respectively.

	As in the discrete case, we define the length of $\alpha^{(*)}_{xy}$ to be $|x-y|$, and we define the length of a composition of these to be the sum of their individual lengths.
        
	By $\Bbbk$-linearizing, for some algebraically closed field $\Bbbk$, we have $\Hom$ spaces given by
	\[
		\Hom_{\preR}(x,y) \cong \Bbbk^{\RR_{\geq |x-y|}},
	\]
	where the $(|x-y|)$th coordinate corresponds to $\alpha^{(*)}_{yx}$.
	That is, $\Hom_{\preR}(x,y)$ contains all finite sums $\sum_{i=1}^m \lambda_if_i$ where each $\lambda_i\in\Bbbk$ and each $f_i:x\to y$ is a composition of $\alpha$ and $\alpha^*$'s.
	
	Also as in the discrete case, we refer to a non-zero composition of $\alpha$'s and $\alpha^*$'s as path like.
	(Over $\RR$, any composition is non-zero, but later we will consider situations where some compositions are zero.)
	We denote by $\len(f)$ the length of a path like morphism in $\preR$.

	Notice that, when $x=y$, we have $\End(x)\cong \Bbbk^{\RR_{\geq 0}}$.
	This acts like a graded ring because we can use the relations 1 and 2 to obtain
	\[
		\alpha^*_{x,x+b}\alpha_{x+b,x}\circ \alpha^*_{x,x+a}\alpha_{x+a,x} = \alpha^*_{x,x+a+b}\alpha_{x+a+b,x}.
	\]
	
	In fact every path like morphism in $\preR$ has a length, or grading, such that the length of $g\circ f$ is equal to the sum of lengths of $f$ and $g$, for path like $f$ and $g$.
	For convenience, we consider $\preR$ to have a zero object (an object that is both initial and terminal in $\preR$). 
	
	
	\subsection{Definition for other intervals}\label{sec:category:definition for I}
	
	In this section we define and examine $\preI$ for an interval $\II\subseteq\RR$.
	
	\begin{definition}\label{def:preI}
		Let $\II$ be an arbitrary subinterval of $\RR$.
		We define $\preI$ to be the unique \emph{quotient} category of $\preR$ obtained by setting $x$ isomorphic to the zero object for each $x\notin\II$.
	\end{definition}
	
	If $\II\subsetneq \RR$ there is a consequence on $\Hom$ spaces.
	First note that this implies $\II$ has a lower bound or upper bound (possibly both).
	If $\II$ has one bound, call it $\bou$, then we let $\boldell_x=|x-\bou|$.
	If $\II$ has both an upper bound and lower bound, say $\bou_0$ and $\bou_1$ respectively, then we let $\boldell_x=\min\{|x-\bou_0|,|x-\bou_1|\}$.
	We denote by $\bou_x$ the closest endpoint to $x$, whether $\II$ has one or two endpoints.
	If $|x-\bou_0|=|x-\bou_1|$, and exactly one of $\bou_0$,$\bou_1$ is not in $\II$, we pick $\bou_x$ to be the one not in $\II$.
	If $|x-\bou_0|=|x-\bou_1|$ and either $\bou_0,\bou_1\in\II$ or $\bou_0,\bou_1\notin\II$, we pick $\bou_x=\bou_0$.
		
	Set $\boldell_{xy}=\min\{\boldell_x,\boldell_y\}$.	
	For a pair $x$ and $y$ in $\II$, let $\bou_{xy}=\bou_x$ if $\boldell_x<\boldell_y$ and let $\bou_{xy}=\bou_y$ if $\boldell_y<\boldell_x$.
	If $\boldell_x=\boldell_y$ then we determine $\bou_{xy}$ as follows:
	\begin{itemize}
		\item If $\bou_x\in\II$ but $\bou_y\notin\II$, then $\bou_{xy}=\bou_y$.
		\item If $\bou_y\in\II$ but $\bou_x\notin\II$, then $\bou_{xy}=\bou_x$.
		\item If $\bou_x,\bou_y\in\II$ or $\bou_x,\bou_y\notin\II$, then $\bou_{xy}=\min\{\bou_x,\bou_y\}$.
	\end{itemize}
	Notice $\bou_{xy}=\bou_{yx}$, $\boldell_{xy}=\boldell_{yx}$, $\bou_{xx}=\bou_x$, and $\boldell_{xx}=\ell_x$.
	
	We can now describe $\Hom$ spaces in $\preI$:
	\[
		\Hom_{\preI}(x,y) \cong
		\begin{cases}
			\Bbbk^{[\,|x-y|,\,|x-y|+\boldell_{xy}\,]} & \bou_{xy}\in\II \\
			\Bbbk^{[\,|x-y|,\,|x-y|+\boldell_{xy}\,)} & \bou_{xy}\notin\II.
		\end{cases}
	\]
	This also gives us $\End_{\preI}(x)\cong \Bbbk^{[\, 0, \boldell_x\, ]}$ or $\End_{\preI}(x)\cong \Bbbk^{[\, 0, \boldell_x\, )}$, depending on if $\bou_x\in\II$ or $\bou_x\notin\II$, respectively.
	Since $\boldell_{\bou_0}=0$, notice that, for $\bou_0,y\in\II$, we have $\Hom_{\preI}(\bou_0,y) \cong \Bbbk$.
	This is similarly true for $\bou_1$, or $\bou$ if $\II$ has just one bound.
	
	Intuitively, one should think of $\bou_x$ being infinitesimally closer to a point $x\in\II$ if $\bou_x\notin\II$ than if $\bou_0\in\II$, although doing this technically would only complicate things.
	
	There is an interesting situation that can happen with $\II$ that cannot happen with finite-dimensional preprojective algebras.
	Typically, for a vertex $i$ at one end of the $A_n$ diagram, there is a unique path in the preprojective path algebra from $i$ to any other vertex $j$.
	In our context, we say that we have a one-dimensional $\Hom$ space, like with $\bou$ above.
	However, it is possible that $\II$ is a bounded open subinterval of $\RR$.
	This yields nontrivial graded endomorphism rings everywhere and always multi-dimensional $\Hom$ spaces.
	
	In our setting, we actually want $\II$ to be open since the point(s) on its boundary do not significantly contribute to $\Hom$ spaces or endomorphism rings.
	Moreover, when we look at permutons in Section~\ref{sec:permutons}, we will see that taking $\II$ to be open removes some technical considerations.
	
	\begin{proposition}\label{prop:preprojective classes}
		Up to equivalence, there are three preprojective categories $\preI$ for an open $\II\subseteq\RR$:
		\begin{enumerate}
			\item\label{prop:preprojective classes:bounded} $\II$ is bounded. That is, $\II=(a,b)$ for $a,b\in\RR$.
			\item\label{prop:preprojective classes:half bounded} $\II$ is bounded on one side but not the other. That is, $\II=(-\infty,b)$ or $\II=(a,+\infty)$, for $a,b\in \RR$.
			\item\label{prop:preprojective classes:unbounded} $\II=\RR$.
		\end{enumerate}
	\end{proposition}
	\begin{proof}
		First we show that for any open, bounded $\II$ and $\JJ$, the categories $\preI$ and $\preJ$ are equivalent.
		For any pair of finite intervals $\II$ and $\JJ$ there is a linear, order preserving homeomorphism $\II\to\JJ$ as intervals that extends to a functor between the corresponding categories.
		Let $r$ be the length of $\JJ$ divided by the length of $\II$.
		Then a path like morphism $f$ in $\preI$ with length $a$ is sent to a path like morphism in $\Lambda_{\JJ}$ with length $ra$.
		This gives us an equivalence of categories.
		
		If $\II$ and $\JJ$ are both bounded above or both bounded below, say by $a$ and $b$, respectively, then we may simply translate $\II$ to $\JJ$ as intervals.
		In particular, we send $x$ in $\II$ to $(b-a)+x$ in $\JJ$.
		This gives us an equivalence of categories between $\preI$ and $\preJ$.
		If $\II$ is bounded below, then $\preI$ is equivalent to $\Lambda_{(0,+\infty)}$.
		Then, multiplying by $-1$, we see $\Lambda_{(0,+\infty)}$ is equivalent to $\Lambda_{(-\infty,0)}$.
		Thus, for all half bounded $\II$ and $\JJ$, the categories $\preI$ and $\preJ$ are equivalent.
		
		Now we show the three cases are distinct.
		We can immediately see that $\preR$ is distinct from the bounded and half-bounded cases because there are no path like morphisms in $\preR$ that compose to 0.
		To see the bounded and half bounded classes are mutually exclusive, let $0<x<z$ in $\II=(0,+\infty)$.
		Choose any $y\in \II$ strictly between $x$ and $z$.
		Then $\Hom_{\preI}(x,z)\to \Hom_{\preI}(y,z)$, given by $f\mapsto f\circ \alpha^*_{xy}$, is always a strict inclusion.
		However, if we have the bounded interval $\JJ=(0,1)$ and choose $x=\frac{1}{2}$, there is no pair $z$ and $y$ such that $\Hom_{\preJ}(x,z)\to\Hom_{\preJ}(y,z)$ is a strict inclusion by any $\alpha^{(*)}$.
		One may use the grading to quickly check that this is true.
		If there is no inclusion with an $\alpha^{(*)}$ then there can be no inclusion by a path like morphism or a finite sum of path like morphisms.
	\end{proof}
	
	So, from now on, we assume $\II=(0,1)$ when $\II$ is open and bounded.
	
	\begin{remark}\label{rmk:preI and preIop}
		For any bounded open interval $\II=(a,b)\subset \RR$, we see that $\preI$ is canonically isomorphic to $\preI^{\text{op}}$ by sending each object $x\in \preI$ to the object $x\in \preI^{\text{op}}$.
		On morphisms we have $\alpha_{yx}\mapsto (\alpha^*_{xy})^{\text{op}}$ and $\alpha^*_{xy}\mapsto (\alpha_{yx})^{\text{op}}$, for $x<y$ in $\II$.
	\end{remark}
        
	\subsection{Representable projectives}\label{sec:category:projectives}
	This section begins the study of $\preI$-modules, for some bounded and open $\II\subsetneq\RR$, starting with representable projectives.
	In particular, we consider our modules to be functors from $\preI$ to $\kvec$.
	In later sections we will consider right $\preI$-modules even though in earlier sections we consider functors written as left modules.
	Because of the canonical isomorphism $\preI\cong\preI^{\text{op}}$ (Remark~\ref{rmk:preI and preIop}) this discrepancy does not actually cause any issues.

	We relate representable projective $\preI$-modules to objects in the bounded derived category $\DbAI$, where $A_{\II}$ is a continuous quiver of type $A$, indexed over $\II$ instead of $\RR$, with finitely-many sinks and sources.
	See \cite{IRT23} for the definition and foundational theorems about representations of continuous quivers of type $A$.
	
	A \emph{$\preI$-module} is a $\Bbbk$-linear functor $\preI\to\Bbbk\text{-}\mathrm{Vec}$.
	Since every $\Hom$ space in $\preI$ is infinite-dimensional, we see immediately that we cannot live in the world of pointwise finite-dimensional modules if we wish to study even the representable projective modules.
	
	Assume $A_{\II}$ has finitely-many sinks and sources.
	Then, by \cite{R20}, the indecomposables in $\mathcal{D}^b(\mathrm{rep}^{fp}(A_{\II}))=:\DbAI$ form an infinite strip that we may index as $\II\times\RR$.
	The category $\DbAI$ is triangulated equivalent to the category $\mathcal{D}$ from \cite{IT15}.
	In both cases, we have reindexed the (indecomposable) objects so that we have an infinite vertical strip.
	
	Consider $P_x:=\Hom_{\preI}(x,-)$
	and let $(x,a)\in\II\times\RR$, for some $a\in\RR$.
	Using \cite{R20,IT15}, we can visualize the $\Hom$ support of $(x,a)$ in $\DbAI$ with Figure~\ref{fig:Hom support of (x,a)}.
	\begin{figure}[h]\begin{center}\begin{tikzpicture}[scale=3]
		\filldraw[fill=blue, fill opacity=.25, draw opacity=0] (.3,1) -- (1,.3) -- (.7,0) -- (0,.7) -- cycle;
		\foreach \x in {0,1}
			\draw[dashed, draw opacity=0.5] (\x,1.1) -- (\x,-0.1);
		\filldraw[blue] (.3,1) circle[radius=.015];
		\draw[blue] (0,.7) -- (.3,1) -- (1,.3);
		\draw[dashed, draw opacity=0.5] (1,.3) -- (.7,0) -- (0,.7);
		\filldraw[fill=white, draw opacity=0.5] (0,.7) circle[radius=.015];
		\filldraw[fill=white, draw opacity=0.5] (1,.3) circle[radius=.015];
		\filldraw[fill=white, draw opacity=0.5] (.7,0) circle[radius=.015];
		\draw (.1,1.1) node[anchor=south east] {\scriptsize $\{0\}\times \RR$};
		\draw (0.9,1.1) node[anchor=south west] {\scriptsize $\{1\}\times \RR$};
		\draw (.3,.99) node[anchor=south] {\scriptsize $(x,a)$};
		\draw[<->, draw opacity =0.5] (-0.5,.9) -- (-0.5,.1);
		\draw (-0.5,.1) node[anchor=north] {\scriptsize $+$};
		\draw (-0.5,.9) node[anchor=south] {\scriptsize $-$};
	\end{tikzpicture}
	\caption{$\Hom$ support of $(x,a)$ in $\DbAI$. In our figures we use the following conventions. The positive direction in $\RR$ is \emph{down}. We use light gray to indicate a line or curve of points is not included. We similarly draw light gray circles with white interior to indicate that a point is not included. We use coloring to indicate a line, curve, or point is included. We use light coloring to indicate a region is included.}\label{fig:Hom support of (x,a)}
	\end{center}\end{figure}
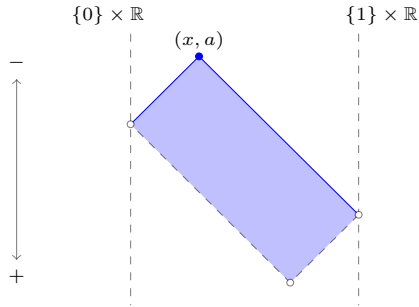
	
	We can visualize $\Hom_{\preI}(x,-)$ as $\Hom_{\DbAI}((x,a),-)$.
	Each path like $f:x\to y$ has a length $\ell$ in $[|x-y|,|x-y|+\boldell_{xy})$.
	The point $(y,a+\ell)$ is inside the $\Hom$ support of $(x,a)$.
	However, if we compose $f$ with some $f':y\to z$ such that $f'\circ f$ would have length $\ell'\geq |x-y|+\boldell_{xy}$, we get the 0 morphism.
	We also notice that, in this case, the point $(z,\ell')$ is not in the $\Hom$ support of $(x,a)$.
	Thus, the $\Hom$-support of $(x,a)$ in $\DbAI$ helps us ``see'' all the path like morphisms $f$ in $P_x$.
	
\section{Submodules and quotient modules of representable projectives}\label{sec:subs and quots of projs}
	Now we want to look at submodules $M$ (and quotient modules $N$ ) of $P_x$ for some $x\in\II$.
	We may also find a submodule $\widetilde{M}$ of $\widetilde{P}_{x}=\Hom_{\DbAI}((x,a),-)$ that helps us model $M$.
	We continue to consider modules as functors $\preI\to\kVec$, which coincides with left modules, but these are equivalent to right modules since $\preI$ is canonically isomorphic to $\preI^{\text{op}}$ (Remark~\ref{rmk:preI and preIop}).
	
	Recall, by Proposition~\ref{prop:preprojective classes}(\ref{prop:preprojective classes:bounded}), that we can, and do, assume that $\II=(0,1)$.
	
	Let $M$ be a module in $\ReppreI$ and $y\in \II$.
	Consider a sum $\sum_{i=1}^m \lambda_i f_i$ of elements of $M(y)$, where each $\lambda_i\in \Bbbk$ and each $f_i$ is path like.
	Unless stated otherwise, we assume $f_i\neq f_j$ if $i\neq j$ and that $\lambda_i\neq 0$ for all $1\leq i\leq m$.
	In particular, we do not have something like $\lambda f_1- \lambda f_1 + 0f_2$ in our sum.
	Moreover, we assume $\len(f_i)<\len(f_j)$ if and only if $i<j$.
	
	\begin{lemma}\label{lem:submodules of projectives are downward closed}
		Suppose $M$ is a submodule of $P_x=\Hom_{\preI}(x,-)$.
		If $\sum_{i=1}^m \lambda_i f_i\in M(y)\subseteq\Hom_{\preI}(x,y)$, for each $\lambda_i\in\Bbbk$ and each $f_i$ path like, then $f_i\in M(y)$ for each $1\leq i\leq m$.
		Moreover, if $f\in\Hom_{\preI}(x,y)$ is path like with length $\ell$ and there is some $g\in M(y)$ with length less than $\ell$, then $f\in M(y)$.
	\end{lemma}
	\begin{proof}
		The proof is by induction of $m$, the number of terms appearing in the sum, with $m=1$ as the trivial base case. While, logically, the $m=2$ case can be treated as a special case of the general $m$ analysis, we write it out explicitly, as a warm-up. 
		Suppose, then, that $\lambda_1 f_1 + \lambda_2 f_2\in M(y)$, where $f_1,f_2$ are path like and $\lambda_1,\lambda_2\in\Bbbk$.
	
		There must be a path like endomorphism $\lambda_3f_3$ in $\End_{\preI}(y)$ of length $\len(f_2)-\len(f_1)$, or else $f_2=0$. 
		Then $\lambda_2 f_2 + \lambda_2\lambda_3 f_3 f_2\in M(y)$. 
		So, we must have 
	\[
		\lambda_1 f_1 + \lambda_2 f_2- \lambda_2f_2 + \lambda_2\lambda_3f_3f_2 = \lambda_1 - \lambda_2 \lambda_3 f_3 f_2 \in M(y).
	\]
		Set $\lambda'_2=-\lambda_2\lambda_3$, $f'_2=f_3f_2$.
	
		Then we start again with $\lambda_1f_1 + \lambda'_2f'_2$ except $\len(f'_2)-\len(f_1)=2(\len(f_2)-\len(f_1))$.
		In finitely-many steps, we obtain $\lambda_1f_1+ \lambda''_2f''_2$ where $2(\len(f''_2)-\len(f_1))> \boldell_{xy}-\ell_1$.
		In this situation, we multiply by $\lambda''_3f''_3\in\End_{\preI}(y)$ whose length is $\len(f''_2)-\len(f_1)$.
		We obtain simply $\lambda''_2f''_2$.
		Then $\lambda_1f_1 +\lambda''_2f''_2 - \lambda''_2f''_2=\lambda_1f_1\in M(y)$.
		Therefore $f_1\in M(y)$ and $f_2=\lambda_2^{-1}(\lambda_1f_1+\lambda_2f_2-\lambda_1f_2)\in M(y)$.

		We now give the general argument for $m\geq 2$. Let us write $F$ for $\sum_{i=1}^m \lambda_if_i$. 
		The strategy is the same: by subtracting from $F$ a multiple of $F$ by an endomorphism of $y$, we can cancel off the next-to-lowest length term in $F$, at the expense of introducing some higher-length terms.
		We seek to do this repeatedly, until the length of the next-to-lowest length term would be greater than $\ell_{xy}$, which would allow us to conclude that there is no next-to-lowest length term, and $\lambda_1f_1\in M(y)$.
		At this point, we know that $F-\lambda_1f_1\in M(y)$, and we can proceed by induction.

		The inconvenience here, compared to the $m=2$ case, is that we cannot initially determine how much we will increase the length of the next-to-lowest length term at each step, and in principle, there could be an infinite sequence of steps yielding smaller and smaller increases.
		It is possible this sequence would never get us to the point where the next-to-lowest length would be greater than $\ell_{xy}$.
		We therefore need a lower bound on the increase at each step.

		Let $R$ be the additive semigroup generated by $\{\len(f_i)-\len(f_1)\}_{2\leq i \leq m}$.
		Clearly, all the lengths appearing in $F$ lie in $\len(f_1)+ R$.
		We claim that this remains true at each step of the process, because we always want to cancel off some length which appears (and which, by assumption, is of the form $\len(f_1)+n$ with $n\in R$).
		We therefore subtract off $gF$, where $g\in\End_{\preI}(y)$ is of length $n$.
		Thus, a term in $gF$ of length $\len(f_1)+n'$, with $n'\in R$, contributes a term of length $\len(f_1)+n'+n$ to the new sum.

		Now observe that there are only finitely many elements of $R$ which are less than $\ell_{xy}-\len(f_1)$.
		Therefore, there is some minimal difference between them, say $\delta$.
		At each step in the process, we increase the length of the next-to-lowest length term of the sum by at least $\delta$.
		Therefore, after a finite number of steps, we will achieve a sum equal to $\lambda_1f_1$, and we can proceed by induction as explained above.

		Finally, since $M$ is a submodule of $P_x$, if $f\in M(y)$ is path like with length $\ell$ then we must have all path like $f'\in M(y)$ where the length of $f'$ is greater than $\ell$ and $f'\in P_x(y)$.
		This completes the proof.
	\end{proof}
	
	Lemma~\ref{lem:submodules of projectives are downward closed} allows us to immediately describe $\widetilde{M}$ as a submodule of $\widetilde{P}_x$ such that $\widetilde{M}$ corresponds to $M$ and $\widetilde{P}_{(x,a)}$ corresponds to $P_x$.
	We take all points $(y,b)$ in the Hom support of $(x,a)$ such that there is a path like morphism of length $b$ in $M(y)$.
	
	We see some examples in Figure~\ref{fig:submodules of projectives}.
	For each point $(y,b)$ in the support of a submodule $\widetilde{M}\hookrightarrow\Hom_{\DbAI}((x,a),-)$, we see that the support of $\widetilde{M}$ must contain the $\Hom$ support of $(x,a)$ intersected with the $\Hom$ support of $(y,b)$.
	Intuitively, submodules of $P_x$ must be ``downward closed'' when drawn as $\DbAI$-modules.
	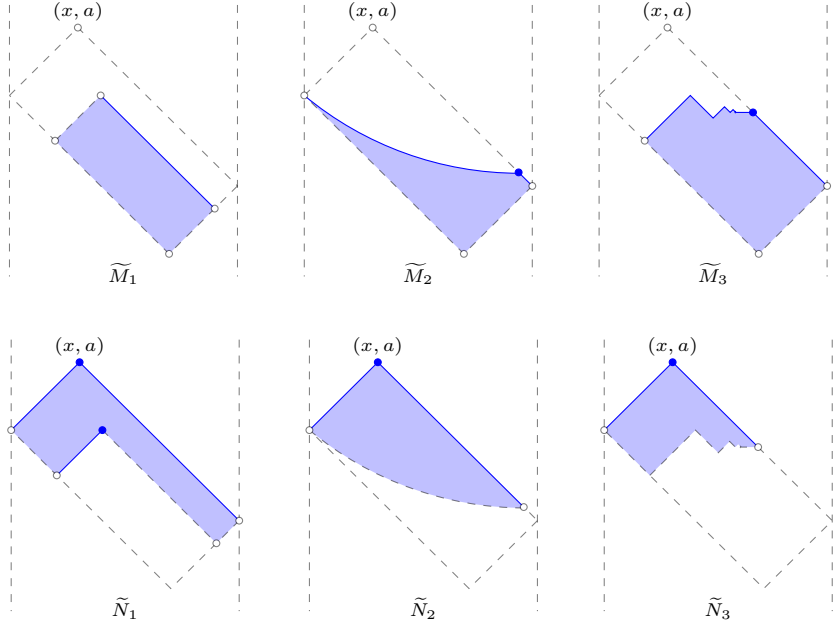
\begin{figure}\begin{center}
	\begin{tikzpicture}[scale=3]
		\filldraw[fill=blue, fill opacity=.25, draw opacity=0] (.7,0) -- (.2,.5) -- (.4,.7) -- (.9,.2) -- cycle;
		\foreach \x in {0,1}
			\draw[dashed, draw opacity=0.5] (\x,1.1) -- (\x,-0.1);
		\draw[dashed, draw opacity=0.5] (.2,.5) -- (.4,.7);
		\draw[blue] (0.4,0.7) -- (.9,.2);
		\draw[dashed, draw opacity=0.5] (0,.7) -- (.3,1) -- (1,.3);
		\draw[dashed, draw opacity=0.5] (1,.3) -- (.7,0) -- (0,.7);
		\filldraw[fill=white, draw opacity=0.5] (.3,1) circle[radius=.015];
		\filldraw[fill=white, draw opacity=0.5] (.7,0) circle[radius=.015];
		\filldraw[fill=white, draw opacity=0.5] (.9,.2) circle[radius=.015];
		\filldraw[fill=white, draw opacity=0.5] (.2,.5) circle[radius=.015];
		\filldraw[fill=white, draw opacity=0.5] (0.4,0.7) circle[radius=0.015];
		\draw (.3,.99) node[anchor=south] {\scriptsize $(x,a)$};
		\draw (0.5,0) node[anchor=north] {\scriptsize $\widetilde{M}_1$};
	\end{tikzpicture}
	\qquad
	\begin{tikzpicture}[scale=3]
		\filldraw[fill=blue, fill opacity=.25, draw opacity=0] (0,.7) arc(230:270:1.47) -- (1,.3) -- (.7,0) -- (0,.7) -- cycle;
		\foreach \x in {0,1}
			\draw[dashed, draw opacity=0.5] (\x,1.1) -- (\x,-0.1);
		\draw[blue] (0,.7) arc(230:270:1.47);
		\draw[dashed, draw opacity=0.5] (0,.7) -- (.3,1) -- (0.94,0.36);
		\draw[blue] (0.94,0.36) -- (1,0.3);
		\draw[dashed, draw opacity=0.5] (1,.3) -- (.7,0) -- (0,.7);
		\filldraw[fill=white, draw opacity=0.5] (.3,1) circle[radius=.015];
		\filldraw[fill=white, draw opacity=0.5] (0,.7) circle[radius=.015];
		\filldraw[fill=white, draw opacity=0.5] (1,.3) circle[radius=.015];
		\filldraw[fill=white, draw opacity=0.5] (.7,0) circle[radius=.015];
		\filldraw[blue] (0.94,0.36) circle[radius=0.015];
		\draw (.3,.99) node[anchor=south] {\scriptsize $(x,a)$};
		\draw (0.5,0) node[anchor=north] {\scriptsize $\widetilde{M}_2$};
	\end{tikzpicture}
	\qquad
	\begin{tikzpicture}[scale=3]
		\filldraw[fill=blue, fill opacity=.25, draw opacity=0] (0.2,0.5) -- (0.4,0.7) -- (0.5,0.6) -- (0.55,0.65) -- (0.575,0.625) -- (0.5875,0.6375) -- (0.59375,.63125) -- (0.596875,0.624475) -- (0.5984375,0.6278625) -- (0.6,.625) -- (.675,.625) -- (1,.3) -- (.7,0) -- (0,.7) -- cycle;
		\foreach \x in {0,1}
			\draw[dashed, draw opacity=0.5] (\x,1.1) -- (\x,-0.1);
		\draw[blue] (0.2,0.5) -- (0.4,0.7) -- (0.5,0.6) -- (0.55,0.65) -- (0.575,0.625) -- (0.5875,0.6375) -- (0.59375,.63125) -- (0.596875,0.624475) -- (0.5984375,0.6278625) -- (0.6,.625) -- (.675,.625);
		\draw[dashed, draw opacity=0.5] (0,.7) -- (.3,1) -- (.675,.625);
		\draw[blue] (.675,.625) -- (1,0.3);
		\draw[dashed, draw opacity=0.5] (1,.3) -- (.7,0) -- (0,.7);
		\filldraw[fill=white, draw opacity=0.5] (.3,1) circle[radius=.015];
		\filldraw[blue] (.675,.625) circle[radius=0.015];
		\filldraw[fill=white, draw opacity=0.5] (0.2,0.5) circle[radius=.015];
		\filldraw[fill=white, draw opacity=0.5] (1,.3) circle[radius=.015];
		\filldraw[fill=white, draw opacity=0.5] (.7,0) circle[radius=.015];
		\draw (.3,.99) node[anchor=south] {\scriptsize $(x,a)$};
		\draw (0.5,0) node[anchor=north] {\scriptsize $\widetilde{M}_3$};
	\end{tikzpicture} \\{~}\\
	\begin{tikzpicture}[scale=3]
		\filldraw[fill=blue, fill opacity=.25, draw opacity=0] (0,0.7) -- (.2,.5) -- (.4,.7) -- (.9,.2) -- (1,0.3) -- (0.3,1) -- cycle;
		\foreach \x in {0,1}
			\draw[dashed, draw opacity=0.5] (\x,1.1) -- (\x,-0.1);
		\draw[blue] (.2,.5) -- (.4,.7);
		\draw[dashed, draw opacity=0.5] (0.4,0.7) -- (.9,.2);
		\draw[blue] (0,.7) -- (.3,1) -- (1,.3);
		\draw[dashed, draw opacity=0.5] (1,.3) -- (.7,0) -- (0,.7);
		\filldraw[blue] (0.4,0.7) circle[radius=0.015];
		\filldraw[blue] (.3,1) circle[radius=.015];
		\filldraw[fill=white, draw opacity=0.5] (0,.7) circle[radius=.015];
		\filldraw[fill=white, draw opacity=0.5] (1,.3) circle[radius=.015];
		\filldraw[fill=white, draw opacity=0.5] (.9,.2) circle[radius=.015];
		\filldraw[fill=white, draw opacity=0.5] (.2,.5) circle[radius=.015];
		\draw (.3,.99) node[anchor=south] {\scriptsize $(x,a)$};
		\draw (0.5,0) node[anchor=north] {\scriptsize $\widetilde{N}_1$};
	\end{tikzpicture}
	\qquad
	\begin{tikzpicture}[scale=3]
		\filldraw[fill=blue, fill opacity=.25, draw opacity=0] (0,.7) arc (230:270:1.47) -- (0.3,1) -- cycle;
		\foreach \x in {0,1}
			\draw[dashed, draw opacity=0.5] (\x,1.1) -- (\x,-0.1);
		\draw[dashed, draw opacity=0.5] (0,.7) arc(230:270:1.47) -- (1,0.3);
		\draw[blue] (0,.7) -- (.3,1) -- (0.94,0.36);
		\draw[dashed, draw opacity=0.5] (1,.3) -- (.7,0) -- (0,.7);
		\filldraw[blue] (.3,1) circle[radius=.015];
		\filldraw[fill=white, draw opacity=0.5] (0,.7) circle[radius=.015];
		\filldraw[fill=white, draw opacity=0.5] (0.94,0.36) circle[radius=.015];
		\draw (.3,.99) node[anchor=south] {\scriptsize $(x,a)$};
		\draw (0.5,0) node[anchor=north] {\scriptsize $\widetilde{N}_2$};
	\end{tikzpicture}
	\qquad
	\begin{tikzpicture}[scale=3]
		\filldraw[fill=blue, fill opacity=.25, draw opacity=0] (0.2,0.5) -- (0.4,0.7) -- (0.5,0.6) -- (0.55,0.65) -- (0.575,0.625) -- (0.5875,0.6375) -- (0.59375,.63125) -- (0.596875,0.624475) -- (0.5984375,0.6278625) -- (0.6,.625) -- (.675,.625) -- (0.3,1) -- (0,0.7) -- cycle;
		\foreach \x in {0,1}
			\draw[dashed, draw opacity=0.5] (\x,1.1) -- (\x,-0.1);
		\draw[dashed, draw opacity=0.5] (0.2,0.5) -- (0.4,0.7) -- (0.5,0.6) -- (0.55,0.65) -- (0.575,0.625) -- (0.5875,0.6375) -- (0.59375,.63125) -- (0.596875,0.624475) -- (0.5984375,0.6278625) -- (0.6,.625) -- (.675,.625);
		\draw[blue] (0,.7) -- (.3,1) -- (.675,.625);
		\draw[dashed, draw opacity=0.5] (.675,.625) -- (1,.3) -- (.7,0) -- (0,.7);
		\filldraw[blue] (.3,1) circle[radius=.015];
		\filldraw[fill=white, draw opacity=0.5] (0,.7) circle[radius=.015];
		\filldraw[fill=white, draw opacity=0.5] (.675,.625) circle[radius=.015];
		\draw (.3,.99) node[anchor=south] {\scriptsize $(x,a)$};
		\draw (0.5,0) node[anchor=north] {\scriptsize $\widetilde{N}_3$};
	\end{tikzpicture}
	\caption{On the top row, some examples of submodules of $\widetilde{P}_x=\Hom_{\DbAI}((x,a),-)$ to help us understand submodules of $P_x=\Hom_{\preI}(x,-)$. On the bottom row, some examples of quotient modules of $\widetilde{P}_x$ to help us understand quotient modules of $P_x$.
	For $i\in\{1,2,3\}$, the sequence $\widetilde{M}_i\hookrightarrow \widetilde{P}_x\twoheadrightarrow \widetilde{N}_i$ is a short exact sequence.}\label{fig:submodules of projectives}
	\end{center}\end{figure}
	
	Either by performing a dual argument to those we have just completed, or by examining $\DbAI$-module quotients of the form $\widetilde{P}_x /\widetilde{M}$, we can see that quotients of a $P_x$ are precisely the complements of the submodules we have described.
	There are examples of quotient modules of $P_x$ also contained in Figure~\ref{fig:submodules of projectives}.
	
	The following is a useful property about submodules and quotient modules of $P_x$, for some $x\in\II$.
	In order to not confuse the notation $\supp$ with ``supremum'', we write $\lub X$ to mean the least upper bound of $X$ (equivalently, the supremum of $X$).
	\begin{lemma}\label{lem:support of sub or quot modules of projectives is path connected}
		If $M$ is a nonzero submodule of $P_x$, then $\supp(M)$ is path connected as a interval in $\II$ and contains $\lub\II-x$.
		Similarly, if $N$ is a nonzero quotient module of $P_x$, then $\supp(N)$ is path connected as an interval in $\II$ and contains $x$.
	\end{lemma}
	\begin{proof}
		We will prove the the statement about submodules as the proof about quotient modules is similar.
		First, let $y\in \II$ such that $M(y)\neq 0$.
		Choose some path like morphism $f\in M(y)$.
		Without loss of generality, suppose $x < \lub\II-x$.		
		
		Suppose $y<w<z$ are points in $\II$ such that $M(y)$ and $M(z)$ are nonzero.
		Let $a_1=|x-y|+\boldell_{xy}$, $a_2=|x-z|+\boldell_{xz}$, and $b=|x-w|+\boldell_{xw}$.
		Then, there is some $\e_1>0$ such that there is a path like morphism in $M(y)$ of length $a_1-\e_1$.
		Similarly, there is some $\e_2$ such that there is a path like morphism in $M(z)$ of length $a_2-\e_2$.
		Let $\e=\min\{\e_1,\e_2\}$.
		Then there is a path like $f_1\in M(y)$ of length $a_1-\e$ and there is a path like $f_2\in M(z)$ of length $a_2-\e$.
		
		There are three cases: (1) $y< z\leq (\lub\II-x)$, (2) $(\lub\II-x)\leq y<z$, and (3) $y<(\lub\II-x)<z$.
		In case (1), we have $f_2= \alpha_{zy} f_1\in M(z)$ and $\alpha_{zy}=\alpha_{zw}\alpha_{wy}$ by our definition of $\II$.
		Thus, $\alpha_{wy}f_1\neq 0$ and $\alpha_{wy}f_1\in M(w)$.
		Similarly, for case (2), $0\neq \alpha^*_{wz}f_2\in M(w)$.
		
		Now consider case (3).
		Either (i) $y<w\leq (\lub\II-x)$ or (ii) $(\lub\II-x)\leq w<z$.
		If (i), then replace $z$ with $z'=\lub\II-x$ since $\alpha^*_{z'z}f_2\neq 0$ in $M(z')$.
		This reduces to case (1).
		If (ii), then replace $y$ with $y'=\lub\II-x$ since $\alpha_{y'y}f_1\neq 0$ in $M(y')$.
		This reduces to case (2).
		Thus, $\supp(M)$ is path connected and so a subinterval of $\II$.
	\end{proof}
	
	Let $x\in\II$.
	Suppose $M$ and $N$ are arbitrary sub- and quotient modules of $P_x$, respectively.
	We now define two functions $\upbnd:\II\cup\{\inf\II,\lub\II\}\to \RR$ and $\dwnbnd:\II\cup\{\inf\II,\lub\II\}\to\RR$ that we will use to define decorous sub- and quotint modules (Definition~\ref{def:decorous}) and sheet modules (Definition~\ref{def:sheet module}).\label{def:boundary functions}
	\begin{displaymath}
		\upbnd(y)= \begin{cases}
			\inf\{\text{length}(f) \mid f\in M(y)\text{ is path like}\} & M(y)\neq 0,y\in\II \\
			\lub\{\text{legnth}(f) \mid f\in \Hom(x,y)\text{ is path like}\} & M(y)=0,y\in\II \\
			x-\inf\II & y=\inf\II \\
			\lub\II -x & y=\lub\II.
		\end{cases}
	\end{displaymath}
	\begin{displaymath}
		\dwnbnd(y)= \begin{cases}
			\lub\{\text{length}(f)\mid f\in N(y)\text{ is path like}\} & N(y)\neq 0,y\in\II \\
			\inf\{\text{length}(f)\mid f\in \Hom(x,y) \text{ is path like}\} & N(y)=0,y\in\II \\
			x-\inf\II & y=\inf\II \\
			\lub\II -x & y=\lub\II.
		\end{cases}
	\end{displaymath}
	
	We now define decorous sub- and quotient modules.
	
	\begin{definition}\label{def:decorous}
		We say a submodule $M$ of $P_x$ is \emph{decorous}\footnote{\textbf{decorous}.\ adj.\ ``following the established traditions of refined society and good taste''. source: Merriam--Webster} if, for each $y\in\II$ such that $M(y)\neq 0$, there exists a path like $f\in M(y)$ of length $\upbnd(y)$.
		
		We say a quotient module $N$ of $P_x$ is \emph{decorous} if, for each $y\in\II$, there does not exist a path like $f\in N(y)$ of length $\dwnbnd(y)$.
	\end{definition}
	
	Notice that, for any $x\in\II$, $P_x$ is both a decorous submodule of itself and a decorous quotient module of itself.
	The 0 module is also always a decorous submodule and decorous quotient module of $P_x$.
	
	In Figure~\ref{fig:submodules of projectives}, the submodules of $P_x$ corresponding to $\widetilde{M}_2$ and $\widetilde{M}_3$ are decorous.
	The quotient modules of $P_x$ corresponding $\widetilde{N}_2$ and $\widetilde{N}_3$ are also decorous.
	However, neither of the $\II$-modules corresponding to $\widetilde{M}_1$ and $\widetilde{N}_1$ are decorous.
	In fact, this pairing always holds.
	
	\begin{proposition}\label{prop:decorous short exact sequence}
		Let $0\to M\to P_x\to N\to 0$ be exact.
		Then $M$ is decorous if and only if $N$ is decorous.
	\end{proposition}
	\begin{proof}
		For any path like $f\in\Hom_{\preI}(x,y)$, since $0\to M\to P_x\to N\to 0$ is exact, $f\in M(y)$ if and only if $f\notin N(y)$.
		We will show that if $0\to M\to P_x\to N\to 0$ is exact then $\upbnd=\dwnbnd$.
		The result then follows.
		
		For any $y\in\II$ and any path like $f\in \Hom_{\preI}(x,y)$ of length greater than $\upbnd(y)$, we know $f\in M(y)$.
		Thus, such an $f$ is not in $N(y)$.
		This shows that $\dwnbnd(y)\leq\upbnd(y)$ for all $y\in\II$.
		
		For contradiction, suppose $\dwnbnd(y) \lneq \upbnd(y)$ for some $y\in \II$.
		Let $f\in\Hom_{\preI}(x,y)$ be path like with length $\ell$ such that $\dwnbnd(y) < \ell < \upbnd(y)$.
		Then $f\notin N(y)$ and $f\notin M(y)$, a contradiction to exactness.
		Thus, $\dwnbnd=\upbnd$.
	\end{proof}
	
	Now we prove a useful fact about $\upbnd$ and $\dwnbnd$.
	
	\begin{proposition}\label{prop:boundaries and decorous modules}
		Let $P_x=\Hom_{\preI}(x,-)$ be an $\II$-module.
		Up to isomorphisms, there is a bijection between decorous submodules of $P_x$ and functions $\partial:\II\cup\{\inf\II,\lub\II\}\to \RR_{>0}$ such that $\partial(\inf\II)=x-\inf\II$, $\partial(\lub\II)=\lub\II-x$, and, for all $y\neq z\in\II$,
		\[
			|\partial(y) - \partial(z)|\leq |y-z|.
		\]
		Moreover, up to isomorphisms, there is a bijection between decorous quotient modules of $P_x$ and such functions.
	\end{proposition}
	\begin{proof}
		By Proposition~\ref{prop:decorous short exact sequence}, we know that decorous submodules of $P_x$ are in bijection with decorous quotient modules of $P_x$, up to isomorphisms.
		Thus, we only prove the bijection regarding submodules.
		Moreover, the statement is easily checked if $M=P_x$ or $M=0$.
		
		Let $M$ be a proper decorous nonzero submodule of $P_x$ with function $\upbnd$.
		We claim that $\upbnd$ is a function as described in the lemma.
		By definition, $\upbnd(\inf\II)=x-\inf\II$ and $\upbnd(\lub\II)=\lub\II-x$.
		Thus, we need only prove the inequality.
		Let $y \neq z\in \II$.
		If $M(y)=0=M(z)$ then we already know the inequality is true by definition.
		Recall that, for each $w\in \II$, 
		\[
			\boldell_{wx}=\lub\{\len(f)\mid f \in\Hom(x,w)\text{ is path like}\}.
		\]
		Define $L(w)=\boldell_{wx}$ and, for completeness, let $L(\inf\II)=\upbnd(x)$ and $L(\lub\II)=\upbnd(\lub\II)$.
		Then we have a continuous function $L:\II\cup\{\inf\II,\lub\II\}\to \RR_{>0}$.
		
		Suppose $M(y)\neq 0\neq M(z)$.
		By Proposition~\ref{lem:support of sub or quot modules of projectives is path connected}, we know that for any $w\in \II$ such that $y<w<z$, we have $M(w)\neq 0$.
		Thus, for all $y\leq w\leq z$, we know $L(w)-\upbnd(w)>0$.
		
		If $y<z<\lub\II-x$, then we know the slop from $(y,L(y))$ to $(z,L(z))$ is $1$.
		Thus, we have $\upbnd(y)<L(z)\leq \upbnd(y)+|y-z|$.
		Next, let $f_y\in M(y)$ be path like with length $\upbnd(y)$.
		Then we know $\alpha^{(*)}_{zy}f_y$ has length $\upbnd(y)+|y-z|$.
		Since $M$ is a submodule of a projective module, we must have $\alpha^{(*)}_{zy}f\neq 0$ in $M(z)$ and so $\upbnd(z)\leq \upbnd(y)+|y-z|$.
		Therefore $|\upbnd(y)-\upbnd(z)|\leq |y-z|$.
		The case where $(\lub\II-x) < y< z$ is similar to this case.
		
		Now we consider if $y<(\lub\II-x) < z$.
		In this case, $\upbnd(y)+|y-z| > L(z)$ and so $\upbnd(z)\leq \upbnd(y)+|y-z|$.
		Similarly, $\upbnd(y) \leq \upbnd(z) + |y-z|$.
		This yields the inequalities
		\begin{align*}
			\upbnd(z) - \upbnd(y) &\leq |y-z| &	\upbnd(y) - \upbnd(z) &\leq |y-z|.
		\end{align*}
		Thus, $|\upbnd(y)-\upbnd(z)|\leq |y-z|$ as desired and so the inequality holds on $\supp(M)$.
		Consequently, $\upbnd$ is continuous on $\supp(M)$.
		
		Since $M$ is nonzero, $\supp(M)\neq\emptyset$.
		Next we prove
		\begin{align*}
			\displaystyle\lim_{y\to\inf\supp(M)}\upbnd(y)&=L(\inf\supp(M)) \\
			\displaystyle\lim_{y\to \lub\supp(M)}\upbnd(y)&=L(\lub\supp(M)).
		\end{align*}
		Since $L$ is continuous, $L$ satisfies the inequality, $\upbnd$ is defined to be $L$ outside the support of $M$, and $\supp(M)$ is connected (Proposition~\ref{lem:support of sub or quot modules of projectives is path connected}), this is sufficient show that $\upbnd$ is continuous and has the proper inequality.
		
		We consider only $\displaystyle\lim_{y\to\inf\supp(M)}\upbnd(y)$ as the consideration of $\displaystyle\lim_{y\to \lub\supp(M)}\upbnd(y)$ is similar.
		We know that if $y\geq\lub\II-x$ such that $y\in\supp(M)$ then, by the same arguments as before, we must have $\lub\II-x\in \supp(M)$.
		Thus, $\inf\II \leq \inf\supp(M) < \lub\II-x$.
		
		For contradiction, suppose $\inf\supp(M)\in\supp(M)$.
		Then we know $\upbnd(\inf\supp(M)) < L(\inf\supp(M))$.
		So there is some $\e>0$ such that $L(\inf\supp(M)-\e) < \upbnd(\inf\supp(M))+\e$.
		Let $f\in M(\inf\supp(M))$ be path like of length $\upbnd(\inf\supp(M))$ and let $f':\inf\supp(M)\to \inf\supp(M)-\e$ be path like in $\II$ with length $\e$.
		So, $f'f\neq 0$ and has length $\upbnd(\inf\supp(M))+\e < L(\inf\supp(M)-\e)$.
		This is a contradiction since $M(\inf\supp(M)-\e)=0$.
		
		Now let $\e>0$ be small and let $\delta>0$ such that $\inf\supp(M)+\delta\in\supp(M)$ and $\delta<\e$.
		Let $y\in\supp(M)$ such that $\inf\supp(M)<y<\inf\supp(M)+\delta$.
		Then $\upbnd(y)<L(y)$ and so $\upbnd(y) - L(\inf\supp(M)) < \delta< \e$, possibly negative.
		
		For contradiction, suppose $\upbnd(y) \leq L(\inf\supp(M)) - \delta$.
		Then let $f\in M(y)$ be the path like with length $\upbnd(y)$.
		Let $f':y\to \inf\supp(M)$ be path like with length $|y-\inf\supp(M)| < \delta$.
		Then $f'f$ is nonzero in $\II$ and thus must be in $M(\inf\supp(M))$, a contradiction since we know $\inf\supp(M)\notin\supp(M)$.
		Then, $L(\inf\supp(M)) -\delta < \upbnd(y) < L(\inf\supp(M)) + \delta$.
		Thus, $|\upbnd(y)-L(\inf\supp(M))|<\delta<\e$.
		Therefore, $\upbnd$ satisfies the desired inequality, is continuous, and is defined as desired at $\inf\II$ and $\lub\II$.
		
		Now suppose $\partial$ is some function as described in the lemma.
		For each $y\in \II$, let
		\[
			M(y) = \Bbbk\langle \text{path like }f\in\Hom_{\preI}(x,y) \mid \partial(y) \leq \len(f) < L(y) \rangle.
		\]
		We will show that this defines a decorous submodule of $P_x$.
		Consider $\widetilde{M}$ defined by $(y,b)\in \supp(\widetilde{M})$ if and only if there is a path like $f\in M(y)$ with length $b$.
		This yields a submodule $\widetilde{M}$ of $\widetilde{P}_{x}=\Hom_{\DbAI}((x,0),-)$.
		By Lemma~\ref{lem:submodules of projectives are downward closed}, this shows that $M$ is indeed a submodule of $P_x$.
		By Definition~\ref{def:decorous}, $M$ is decorous.
		Thus, given a function $\partial$ with the desired properties, we obtain a decorous submodule $M$ of $P_x$.
		
		Notice that if $\partial\neq \partial'$ then there is some $y\in\II$ such that $\partial(y)<\partial'(y)$ or $\partial(y)>\partial'(y)$.
		Then the $M(y)$ we construct and the $M'(y)$ we construct will differ.
		Similarly, if $M\not\cong M'$ are submodules of $P_x$, then there must be some $y\in \II$ such that $M(y)\neq M'(y)$.
		Since both $M$ and $M'$ are submodules, this forces $\upbnd(y)\neq \upbnd'(y)$.
		Thus, we indeed have a bijection.
	\end{proof}
	
	In the following proposition, we use that $\II=(0,1)$ for ease of notation.
	\begin{proposition}\label{prop:boundary functions determine decorous submodules}
		Let $\partial:[0,1]\to\RR$ be a function such that $|\partial(y)-\partial(z)|\leq |y-z|$, for all $y,z\in[0,1]$, and $\partial(1) \neq \partial(0) \pm 1$.
		Then $\partial$ uniquely determines a decorous submodule of $P_x$ where $x=\frac{1}{2}(1+\partial(0)-\partial(1))$.
	\end{proposition}
	\begin{proof}
		First let $\partial:[0,1]\to \RR$ be a function satisfying the hypotheses in the first half of the proposition.
		Let $x=\frac{1}{2}(1+\partial(0)-\partial(1))$ and let $\partial'(y)=\partial(y)-\partial(0)+x$, for all $y\in[0,1]$.
		In particular, $\partial'(0)=x$ and $\partial'(1)=1-x$.
		By Proposition~\ref{prop:boundaries and decorous modules}, this uniquely determines a submodule $M$ of $P_x$.
	\end{proof}
	
	Now we have the following corollary to use later on.
	
	\begin{corollary}\label{cor:permuton functions and decorous submodules}
		Let $\boldsymbol{\partial}$ be the set of all $\partial:[0,1]\to\RR$ such that $|\partial(y)-\partial(z)|\leq |y-z|$, for all $y,z\in[0,1]$, and $\partial(1) \neq \partial(0) \pm 1$.
		Then there is a bijection
		\[
			\xymatrix{
				\boldsymbol{\partial}/\{\partial\sim\partial' \text{ if } (\exists a\in\RR) (\forall y\in[0,1]) (\partial(y)=\partial'(y)+a)\}
				\ar@{<->}[d] \\
				\{\text{decorous submodules of representable projectives}\}.
			}
		\]
	\end{corollary}
	\begin{proof}
		Combine Propositions~\ref{prop:boundaries and decorous modules}~and~\ref{prop:boundary functions determine decorous submodules}.
	\end{proof}
	
	In light of the corollary, we note that the natural representative of each equivalence class is the one where $\partial(0)+\partial(1)=1$, for the following reason.
	Consider the submodule $M$ determined by such a function.
	Then, the function $\upbnd$ determined by $M$, using Proposition~\ref{prop:boundaries and decorous modules}, is precisely $\partial$.
	
\section{Sheet modules and their $\Hom$ spaces}\label{sec:sheets}	
	
	In this subsection we define sheet modules and study the $\Hom$ spaces between them.
	This is the beginning of understanding more general $\Hom$ spaces between $\preI$-modules.
	We continue to consider modules as functors $\preI\to\kVec$, which coincides with left modules, but these are equivalent to right modules since $\preI$ is canonically isomorphic to $\preI^{\text{op}}$ (Remark~\ref{rmk:preI and preIop}).

	\begin{definition}\label{def:sheet module}
		A \emph{sheet module} is an $\preI$-module isomorphic to the image of a composition $M\hookrightarrow P_x\twoheadrightarrow N$ where
		\begin{itemize}
			\item $M\hookrightarrow P_x$ is the canonical inclusion,
			\item $P_x\twoheadrightarrow N$ is the canonical projection, and
			\item both $M$ and $N$ are decorous (Definition~\ref{def:decorous}).
		\end{itemize}
	\end{definition}
	For technical reasons, we allow the $0$ module to be a sheet module.
	The reason we call these sheet modules is that they are a general form of modules constructed from decorous modules that can easily be reinterpreted as a thin $\DbAI$-modules.
	
	First we will describe sheet modules as $\preI$-modules and then we will reinterpret them as $\DbAI$-modules.
	When the sheet module is the $0$ module this is trivial
	.
	Thus, we assume the sheet module is nonzero.
	
	Let $M$ be a submodule of $P_x$ and $N$ a quotient module of $P_x$ such that the composition of the canonical inclusion and projection $M\hookrightarrow P_x \twoheadrightarrow N$ is nonzero.
	Let $\upbnd$ and $\dwnbnd$ be as defined on page~\pageref{def:boundary functions}.
	Let $S$ be the image of this composition.
	
	Let $y\in\II$ and suppose $\upbnd(y)<\dwnbnd(y)$.
	If $f\in \Hom_{\preI}(x,y)$ is a path like morphism of length $\ell$ such that $\upbnd(y)\leq\ell<\dwnbnd(y)$, then $f\in S(y)$ by definition.
	If $\upbnd(y)\geq \dwnbnd(y)$ then $S(y)=0$.
	
	Given a choice of lift $\widetilde{P}_x$ of $P_x$, we can draw $\widetilde{S}$ as the overlap in support of $\widetilde{M}$ and $\widetilde{N}$.
	This can be seen in Figure~\ref{fig:sheet}.
	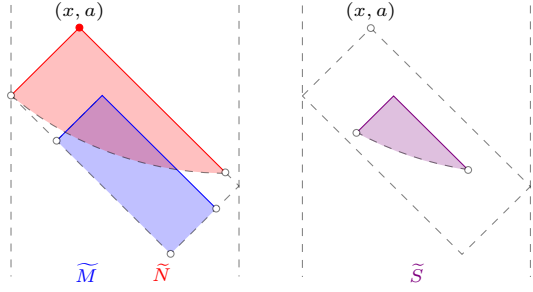
\begin{figure}\begin{center}\begin{tikzpicture}[scale=3]
		\filldraw[fill=blue, fill opacity=.25, draw opacity=0] (.7,0) -- (.2,.5) -- (.4,.7) -- (.9,.2) -- cycle;
		\filldraw[fill=red, fill opacity=.25, draw opacity=0] (0,.7) arc (230:270:1.47) -- (0.3,1) -- cycle;
		\foreach \x in {0,1}
			\draw[dashed, draw opacity=0.5] (\x,1.1) -- (\x,-0.1);
		\draw[dashed, draw opacity=0.5] (.2,.5) -- (.4,.7);
		\draw[blue] (0.2,0.5) -- (0.4,0.7) -- (.9,.2);
		\draw[dashed, draw opacity=0.5] (0,.7) arc (230:270:1.47);
		\draw[red] (0,.7) -- (.3,1) -- (0.94,0.36);
		\draw[dashed, draw opacity=0.5] (0.94,0.36) -- (1,.3) -- (.7,0) -- (0,.7);
		\filldraw[red] (.3,1) circle[radius=.015];
		\filldraw[fill=white, draw opacity=0.5] (.7,0) circle[radius=.015];
		\filldraw[fill=white, draw opacity=0.5] (.9,.2) circle[radius=.015];
		\filldraw[fill=white, draw opacity=0.5] (.2,.5) circle[radius=.015];
		\filldraw[fill=white, draw opacity=0.5] (0,.7) circle[radius=.015];
		\filldraw[fill=white, draw opacity=0.5] (0.94,0.36) circle[radius=.015];
		\draw (.3,.99) node[anchor=south] {\scriptsize $(x,a)$};
		\draw[blue] (0.33,0) node[anchor=north] {\scriptsize $\widetilde{M}$};
		\draw[red] (0.66,0) node[anchor=north] {\scriptsize $\widetilde{N}$};
	\end{tikzpicture}
	\qquad
	\begin{tikzpicture}[scale=3]
		\foreach \x in {0,1}
			\draw[dashed, draw opacity=0.5] (\x,1.1) -- (\x,-0.1);
		\draw[dashed, draw opacity=0.5] (0,0.7) -- (0.3,1) -- (1,0.3) -- (0.7,0) -- (0,0.7);
		\begin{scope}
			\clip (0,.7) arc (230:270:1.47) -- (.4,.7)--(0,.7);
			\clip (0.2,0.5) -- (0.4,0.7) -- (0.9,0.2) -- (0.4,0) -- (0.2,0.5);
			\filldraw[fill = red!50!blue, fill opacity =0.25, draw opacity =0] (0,0) -- (0,1) -- (1,1) -- (1,0) -- cycle;
			\draw[dashed, draw opacity=0.5, thick] (0,.7) arc (230:270:1.47);
		\end{scope}
		\draw[red!50!blue] (0.237,0.537) -- (0.4,0.7) -- (0.727,0.373);
		\filldraw[fill=white, draw opacity =.5] (0.237,0.537) circle[radius=0.015];
		\filldraw[fill=white, draw opacity =.5] (0.727,0.373) circle[radius=0.015];
		\filldraw[fill=white, draw opacity =.5] (0.3,1) circle[radius=0.015];
		\draw (.3,.99) node[anchor=south] {\scriptsize $(x,a)$};
		\draw[red!50!blue] (0.5,0) node[anchor=north] {\scriptsize $\widetilde{S}$};
	\end{tikzpicture}
	\caption{In blue on the left, the supprt of the submodule $\widetilde{M}$ of $\widetilde{P}_x=\Hom_{\DbAI}((x,a),-)$. In red on the left, the supprt of the quotient module $\widetilde{N}$ of $\widetilde{P}_x$. The purple overlap is the support of $\widetilde{S}$. On the right, we draw the supprt of $\widetilde{S}$ in purple.}\label{fig:sheet}
	\end{center}\end{figure}
	
	\begin{remark}\label{rmk:sheets may be decomposable}
		A sheet may not be indecomposable.
		Consider the example where $x=\frac{\lub\II-\inf\II}{2}$, $M$ is determined by two curves, and $N$ is determined by 1 minus those curves.
		Then the sheet module constructed from $M$ and $N$ can be written as the direct sum of at least two non-isomorphic nonzero summands.
		The picture is in Figure~\ref{fig:decomposable sheet}.
		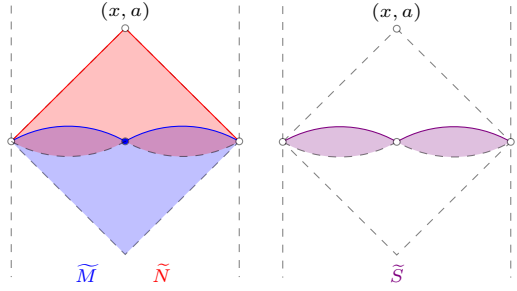
\begin{figure}\begin{center}
		\begin{tikzpicture}[scale=3]
			\foreach \x in {0,1}
				\draw[dashed, draw opacity=0.5] (\x,1.1) -- (\x,-0.1);
			\draw[dashed, draw opacity=0.5] (0,0.5) -- (0.5,1) -- (1,0.5) -- (0.5,0) -- (0,0.5);
			\filldraw[blue, draw opacity=0, fill opacity=0.25] (0,0.5) arc (120:60:0.5) arc (120:60:0.5) -- (0.5,0) -- cycle;
			\filldraw[red, draw opacity=0, fill opacity=0.25] (0,0.5) arc (240:300:0.5) arc (240:300:0.5) -- (0.5,1) -- cycle;
			\draw[red] (0,0.5) -- (0.5,1) -- (1,0.5);
			\draw[blue] (0,0.5) arc (120:60:0.5) arc (120:60:0.5);
			\filldraw[fill=blue, draw opacity =.5] (0.5,0.5) circle[radius=0.015];
			\foreach \x in {0 , 1}
				\filldraw[fill=white, draw opacity=0.5] (\x,0.5) circle[radius=0.015];
			\draw[draw opacity=0.5, dashed] (0,0.5) arc (240:300:0.5) arc (240:300:0.5);
			\filldraw[fill=white, draw opacity=0.5] (0.5,1) circle[radius=0.015];
			\draw (.5,.99) node[anchor=south] {\scriptsize $(x,a)$};
			\draw[blue] (0.33,0) node[anchor=north] {\scriptsize $\widetilde{M}$};
			\draw[red] (0.66,0) node[anchor=north] {\scriptsize $\widetilde{N}$};
		\end{tikzpicture}
		\quad
		\begin{tikzpicture}[scale=3]
			\foreach \x in {0,1}
				\draw[dashed, draw opacity=0.5] (\x,1.1) -- (\x,-0.1);
			\draw[dashed, draw opacity=0.5] (0,0.5) -- (0.5,1) -- (1,0.5) -- (0.5,0) -- (0,0.5);
			\filldraw[blue!50!red, fill opacity =0.25, draw opacity=0] (0,0.5) arc (120:60:0.5) arc (120:60:0.5) arc (300:240:0.5) arc (300:240:0.5);
			\draw[red!50!blue]  (0,0.5) arc (120:60:0.5) arc (120:60:0.5);
			\draw[draw opacity=0.5, dashed] (0,0.5) arc (240:300:0.5) arc (240:300:0.5);
			\foreach \x in {0, 0.5, 1}
				\filldraw[fill=white, draw opacity=0.5] (\x,0.5) circle[radius=0.015];
			\filldraw[fill=white, draw opacity=0.5] (0.5,1) circle[radius=0.015];
			\draw (.5,.99) node[anchor=south] {\scriptsize $(x,a)$};
			\draw[red!50!blue] (0.5,0) node[anchor=north] {\scriptsize $\widetilde{S}$};
		\end{tikzpicture}
		\caption{An example of a decomposable sheet $S$ constructed from $M$ and $N$ and lifted to $\widetilde{S}$, $\widetilde{M}$, and $\widetilde{N}$. The left and right bubbles are direct summands of $\widetilde{S}$ and show us that $S$ decomposes into a left piece and right piece.}\label{fig:decomposable sheet}
		\end{center}\end{figure}
	\end{remark}
	
	For each sheet module $S$, we wish to define a subset $\gen(S)\subset\supp(M)$ of points such that, for every element $\sum_{i=1}^m \lambda_i f_i$ in $S$, we have $f_i=gf_y$ for some morphism $g$ in $\preI$ and some $y\in\gen(S)$.
	We think of $\gen(S)$ as the generators of $S$.
	
	\begin{notation}\label{note:f sub y}
		Let $S$ be a sheet module constructed from $M\hookrightarrow P_x\twoheadrightarrow N$.
		For each $y\in\II$ such that $M(y)\neq 0$, let $f_y$ be the path like element of $S(y)$ whose length is $\upbnd(y)$.
	\end{notation}
	
	\begin{definition}\label{def:gen S}
		We say $y\in\gen(S)\subseteq\supp(S)$ if and only if, for all $z\in\supp(S)\setminus\{y\}$, we have $|y-z|>\upbnd(y)-\upbnd(z)$.
		In particular, this means that there is no $z\in\II$ and no morphism $g:z\to y$ in $\preI$ such that $f_y=gf_z$.
	\end{definition}
	
	For a projective $P_x=\Hom_{\preI}(x,-)$, we know $\gen(P_x)=\{x\}$.
	
	Given two sheet modules $S$ and $S'$, we wish to understand $\Hom_{\ReppreI}(S,S')$.
	To do this, we will construct some auxiliary sets.
	
	Suppose $S$ comes from $M\hookrightarrow P_x\twoheadrightarrow N$ and $S'$ comes from $M'\hookrightarrow P_{x'}\twoheadrightarrow N'$.
	Recall we also have the functions $\upbnd$ and $\dwnbnd$ for $S$ and $\upbnd'$ and $\dwnbnd'$ for $S'$.
	
	We first fix a real number $a$.
	Next, we define a set $C_a(y)\subset\II\times\RR$ for each $y\in\supp(S)$.
	This can be though of as a cone from $(y,a)$ intersected with support of a lift $\widetilde{S}'$ of $S'$ constructed from $\widetilde{M}\hookrightarrow\widetilde{P}_{(x',0)}\twoheadrightarrow \widetilde{N}$.
	We say $(z,b)\in C_a(y)$ if there exists a path like $g:x'\to z$ in $\II$ such that
	\begin{enumerate}
		\item $g\in S'(z)$,
		\item $g$ has length $b$,
		\item and $b-(a+\upbnd(y)) \geq |y-z|$. (Notice the lack of absolute value on the left.)
	\end{enumerate}
	In particular, $(y,a+\upbnd(y))\in C_a(y)$ if there is a path like morphism of length $a+\upbnd(y)$ in $S'(y)$.
	Furthermore, $(y,b)\notin C_a(y)$ if $b< a+\upbnd(y)$.
	A picture of example $C_a(y)$'s can be seen in Figure~\ref{fig:positive cone}.
	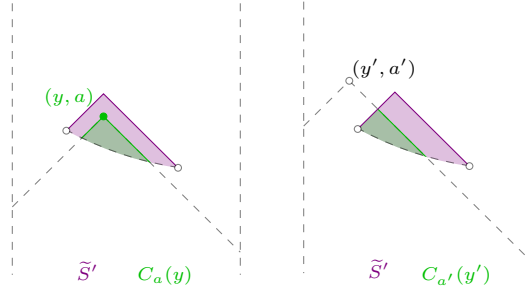
\begin{figure}\begin{center}
	\begin{tikzpicture}[scale=3]
		\foreach \x in {0,1}
			\draw[dashed, draw opacity=0.5] (\x,1.1) -- (\x,-0.1);
		\draw[dashed, draw opacity =0.5] (0,.2) -- (.4,.6) -- (1,0);
		\begin{scope}
			\clip (0,.7) arc (230:270:1.47) -- (.4,.7)--(0,.7);
			\clip (0.2,0.5) -- (0.4,0.7) -- (0.9,0.2) -- (0.4,0) -- (0.2,0.5);
			\filldraw[fill = red!50!blue, fill opacity =0.25, draw opacity =0] (0,0) -- (0,1) -- (1,1) -- (1,0) -- cycle;
			\filldraw[fill = green!75!black, fill opacity=0.25, draw =green!75!black] (.4,.6) -- (-.2,0) -- (1.,0) -- cycle;
			\draw[dashed, draw opacity=0.5, thick] (0,.7) arc (230:270:1.47);
		\end{scope}
		\draw[red!50!blue] (0.237,0.537) -- (0.4,0.7) -- (0.727,0.373);
		\filldraw[fill=white, draw opacity =.5] (0.237,0.537) circle[radius=0.015];
		\filldraw[fill=white, draw opacity =.5] (0.727,0.373) circle[radius=0.015];
		\filldraw[fill=green!75!black, draw=green!75!black] (0.4,0.6) circle[radius=0.015];
		\draw[red!50!blue] (0.33,0) node[anchor=north] {\scriptsize $\widetilde{S}'$};
		\draw[green!75!black] (.67,-.02) node[anchor=north] {\scriptsize $C_a(y)$};
		\draw[green!75!black] (.4,.6) node[anchor=south east] {\scriptsize $(y,a)$};
	\end{tikzpicture}
	\qquad
	\begin{tikzpicture}[scale=3]
		\foreach \x in {0,1}
			\draw[dashed, draw opacity=0.5] (\x,1.1) -- (\x,-0.1);
		\draw[dashed, draw opacity=0.5] (0,.55) -- (.2,.75) -- (1,-0.05);
		\begin{scope}
			\clip (0,.7) arc (230:270:1.47) -- (.4,.7)--(0,.7);
			\clip (0.2,0.5) -- (0.4,0.7) -- (0.9,0.2) -- (0.4,0) -- (0.2,0.5);
			\filldraw[fill = red!50!blue, fill opacity =0.25, draw opacity =0] (0,0) -- (0,1) -- (1,1) -- (1,0) -- cycle;
			\filldraw[fill = green!75!black, fill opacity=0.25, draw =green!75!black, draw opacity =1] (-0.55,0) -- (.2,.75) -- (0.95,0) -- cycle;
			\draw[dashed, draw opacity=0.5, thick] (0,.7) arc (230:270:1.47);
		\end{scope}
		\draw[red!50!blue] (0.237,0.537) -- (0.4,0.7) -- (0.727,0.373);
		\filldraw[fill=white, draw opacity =.5] (0.237,0.537) circle[radius=0.015];
		\filldraw[fill=white, draw opacity =.5] (0.2,0.75) circle[radius=0.015];
		\filldraw[fill=white, draw opacity =.5] (0.727,0.373) circle[radius=0.015];
		\draw[red!50!blue] (0.33,0) node[anchor=north] {\scriptsize $\widetilde{S}'$};
		\draw[green!75!black] (.67,-.02) node[anchor=north] {\scriptsize $C_{a'}(y')$};
		\draw (.17,.72) node[anchor=south west] {\scriptsize $(y',a')$};
	\end{tikzpicture}
	\caption{In purple we have the support of the $\DbAI$-module $\widetilde{S}'$, a lift of $S'$. In green we have the set $C_a(y)$ on the left and the set $C_{a'}(y')$ on the right, which can be seen as a positive cone from $(y,a)$  and from $(y',a')$, respectively, intersected with the support of $S'$. The left $C_a(y)$ contains $(y,a)$ while the right $C_{a'}(y')$ does not contain $(y',a')$.}\label{fig:positive cone}
	\end{center}\end{figure}
	
	\begin{lemma}\label{lem:intersecting positive cones}
		Let $\phi:S\to S'$ be a morphism in $\ReppreI$, where $S$ and $S'$ are sheets, such that $\phi(f_y)$ is path like with length $a+\upbnd(y)$ and $\phi(f_z)$ is path like with length $a+\upbnd(z)$, for some $a\in\RR$.	
		Suppose $(y,a+\upbnd(y))\in C_a(y)$ and $(z,a+\upbnd(z))\in C_a(z)$.
		Then, for each $(w,b)\in C_a(y)\cap C_a(z)$, there exists morphisms $g_1:y\to w$ and $g_2:z\to w$ such that $\phi(g_1f_y)=\phi(g_2f_z)$ and the length of $\phi(g_1f_y)$ is $b$.
	\end{lemma}
	\begin{proof}
		If $C_a(y)\cap C_a(z)=\emptyset$ we are done.
		So, assume $C_a(y)\cap C_a(z)$ is nonempty and let $(w,b)\in C_a(y)\cap C_a(z)$.
		Then there is a path like $g:x'\to w$ in $S'(w)$ with length $b$ such that
		\[
			b-(a+\upbnd(y)) \geq |y-w|
			\qquad \text{and} \qquad
			b-(a+\upbnd(z)) \geq |z-w|.
		\]
		Then there are $g_1:y\to w$ and $g_2:z\to w$ in $\preI$ such that $0\neq g_1f_1=g_2f_2=g$.
		In particular, $\phi(g_1f_y)=g_1f_1=g=g_2f_2=\phi(g_2f_z)$.
	\end{proof}
	
	In Lemma~\ref{lem:intersecting positive cones} we can see that the choice of $\phi(f_y)$ and $\phi(f_z)$ depend on each other.
	E.g., if we defined $\phi':S\to S'$ to have $\phi'(f_y)=\lambda \phi(f_y)$ then we must have $\phi'(f_z)=\lambda \phi(f_z)$.
	This leads us to the following definition.
	
	\begin{definition}\label{def:codependent}
		Let $y,z\in\II$, $a\in\RR$, and $S,S'$ be sheets in $\ReppreI$.
		We say $y$ and $z$ are \emph{$a$-codependent} when $\phi(f_y)$ has (a multiple of) a path like summand with length $a+\upbnd(y)$ and only if $\phi(f_z)$ has (a multiple of) a path like summand with length $a+\upbnd(z)$, for any morphism $\phi:S\to S'$.
	\end{definition}

	Notice that in the context of Lemma~\ref{lem:intersecting positive cones}, we have $y$ and $z$ are $a$-codependent.
	
	Given a $y\in\gen(S)$, we wish to construct a maximally codependent set, $\codep_{y,a}\subseteq\gen(S)$.
	To do this, let $\Delta:\II\cup\{\inf\II,\lub\II\}\to \RR$ be an auxiliary function given by $\Delta(y)=\dwnbnd'(y)-(a+\upbnd(y))$.
	
	By Proposition~\ref{prop:boundaries and decorous modules}, we know that both $\dwnbnd'$ and $\upbnd$ are continuous and so $\Delta$ is continuous.
	In particular, if $\Delta(y)>0$ then there exists a an open interval $B\subseteq\II$ containing $y$ such that $\Delta(z)>0$ for all $z\in B$.
	Since $\II$ is itself bounded there must be a largest open interval $B_a(y)$\label{def:Bay} containing $y$.
	
	\begin{lemma}\label{lem:a-codependent}
		Let $S$ and $S'$ be sheets in $\ReppreI$, let $y\in \gen(S)\cap\supp(S')$ such that $(y,a+\upbnd(y))\in C_a(y)$, and let $z\in B_a(y)$.
		Then $y$ and $z$ are $a$-codependent.
	\end{lemma}
	\begin{proof}
		We know $\Delta(y)>0$ since $y\in C_a(y)$.
		If $|y-z| < \min\{\Delta(y),\Delta(z)\}$, then $C_a(y)\cap C_a(z)\neq\emptyset$ and by Lemma~\ref{lem:intersecting positive cones} we see $y$ and $z$ are codependent.
		
		Suppose $|y-z|$ is greater than at least one of $\Delta(y)$ or $\Delta(z)$.
		Since $\Delta$ is continuous and $B_a(y)$ is by definition the largest open interval containing $y$ such that $\Delta$ is positive on $B_a(y)$, we know that we have $\displaystyle\lim_{w'\to w} \Delta(w') > 0$, for all $w\in B_a(y)$.
		
		Without loss of generality, assume $y<z$.
		Since $\Delta$ is continuous on the closed interval $[y,z]\subsetneq B_a(y)$, the function $\Delta$ must attain a minimum and maximum on $[y,z]$.
		Let $\e$ be the minimum.
		Then for any $w,w'\in[y,z]$ such that $0<w'-w< \e$, we have $C_a(w)\cap C_a(w')\neq\emptyset$.
		Let $n\in\NN_{>0}$ such that $\frac{z-y}{n} < \e$.
		Let $w_0=y$, $w_n=z$, and for each $1\leq i \leq n-1$, let $w_i=y+i(\frac{z-y}{n})$.
		Then, for each $0\leq i \leq n-1$, we have $C_a(w_i)\cap C_a(w_{i+1})\neq\emptyset$.
		
		Thus, $y=w_0$ and $w_1$ are $a$-codependent, $w_1$ and $w_2$ are $a$-codependent, and so  on up to $w_{n-1}$ and $w_n=z$ are $a$-codependent.
		Since the definition of $a$-codependent is an if and only if statement, we see $y$ and $z$ are $a$-codependent.
	\end{proof}
	
	So, we define $\codep_{y,a}:=B_a(y)\cap\gen(S)$.
	Notice that for any other $z\in B_a(y)\cap\gen(S)$, we have $\codep_{z,a}=\codep_{y,a}$.
	
	\begin{lemma}\label{lem:codep determines a morphism}
		Let $y\in \gen(S)\cap\supp(S')$ and let $f_0$ be a nonzero element of $S'(y)$.
		Let $a\in \RR$ such that the length of $f_0$ is $a+\upbnd(y)$ and assume $C_a(y)\ni (y,a+\upbnd(y))$.
		Finally, assume that, for each $z\in C_a(y)$, we have $\upbnd'(z)\leq \upbnd(z) < \dwnbnd'(z) \leq \dwnbnd(z)$.
		Then there is a morphism $\phi:S\to S'$ determined by $\phi(f_y)=f_0$.
	\end{lemma}
	
	To prove the lemma wet set the notation that $\len(f)$ means the length of $f$, for some path like morphism in $\preI$.
	
	\begin{proof}
		We define $\phi$ explicitly.
		Define $\phi(f_y)=f_0$.
		Let $f\in S(z)$ be path like for some $z\notin B_a(y)$.
		Set $\phi(f)=0$.
		
		Let $z\in B_a(y)$.
		Define $\phi(f_z)$ to be the path like element of $S'(z)$ such that the length of $\phi(f_z)$ is $a+\upbnd(z)$.
		For each path like $f\in S(z)$ we define $\phi(f)$ as follows:
		\[
			\phi(f) = \begin{cases}
				g\phi(f_z) & \len(f) < \len(f_z)+\Delta(z) \text{ and } \len(g)=\len(f)-\len(f_z) \\
				0 & \len(f) \geq \len(f_z)+\Delta(z).
			\end{cases}
		\]
		By assumption, if $(w,b)\in C_a(z)$, for some $z\in B_a(y)$, then there is a path like $f$ of length $b-a$ in $S(w)$ and a path like $g:z\to w$ with $\len(g)=(b-a)-\upbnd(z)$ such that $gf_z=f$.
		I.e., if $(w,b)\in C_a(z)$ then there is some $f\in S(w)$ such that $\phi(f)\neq 0$ has length $b$.
		
		We need now only show that, for each path like $f\in S$ and path like morphism $g:z\to w$ in $\preI$, we have $(S'(g)\circ \phi)(f) = (\phi\circ S(g))(f)$.
		By definition of a sheet module, $S(g)(f) = gf$ and $S'(g)(f') = gf$, for $f,f',g$ path like.
		Let $f\in S(z)$ be path like and suppose $z\notin B_a(y)$.
		Then $\len(\alpha^{(*)}_{wz} f)\geq \len(f_w)+\Delta(w)$ for any $w\in B_a(y)$, by construction.
		Thus, for $f\in S(z)$ and $g$ in $\preI$ path like, and any $w\in \II$, we see $S'(g)(\phi(f)) = 0 = \phi(S(g)(f))$.
		
		Now suppose $f\in S(z)$ is path like, for $z\in B_a(y)$, and $\len(f) < \len(f_z)+\Delta(z)$.
		For any $w\notin B_a(y)$, we have	$S'(\alpha^{(*)}_{wz})(\phi(f)) = 0 = \phi(S(\alpha^{(*)}_{wz}f))$.
		For $w\in B_a(y)$ and $g:z\to w$ path like, if $\len(gf) \geq \len(f_w)+\Delta(w)$ then $\phi(gf)=0$ and so we have $S'(\alpha^{(*)}_{wz})(\phi(f)) = 0 = \phi(S(\alpha^{(*)}_{wz}f))$ again.
		If $\len(gf) < \len(f_w)+\delta(w)$ then $\phi(gf)=g\phi(f)$ and so $S'(g)(\phi(f))=\phi(S(g)(f))$.
	\end{proof}
	
	Let $S,S'$ be sheets in $\ReppreI$.
	As a consequence of the arguments in the proof of the lemma, if there is some $z\in\codep_{y,a}$ such that $\upbnd(z)+a < \upbnd'(z)$, then we must have $\phi(f_z)=0$ for all $z\in\codep_{y,a}$.
	Nonetheless, Lemma~\ref{lem:codep determines a morphism} gives us the first fact necessary for understanding $\Hom_{\ReppreI}(S,S')$.
	
	However, there are complications.
	Consider the sheets $S$ and $S'$ as shown in Figure~\ref{fig:complicated sheets}.
	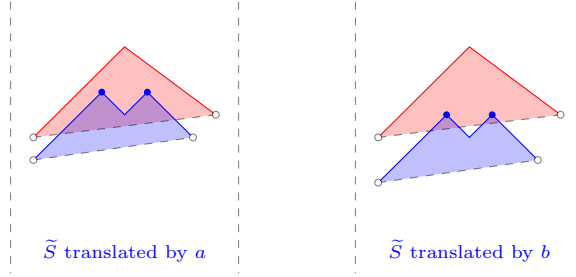
\begin{figure}\begin{center}
	\begin{tikzpicture}[scale=3]
		\foreach \x in {0,1}
			\draw[dashed, draw opacity=0.5] (\x,1.1) -- (\x,-0.1);
		\filldraw[red, fill opacity=.25, draw opacity=0] (.1,.5) -- (.5,.9) -- (.9,.6) -- cycle;
		\filldraw[blue, fill opacity=.25, draw opacity=0] (.1,.4) -- (.4,.7) -- (.5,.6) -- (.6,.7) -- (.8,.5) -- cycle;
		\draw[red] (.1,.5) -- (.5,.9) -- (.9,.6);
		\draw[blue]  (.1,.4) -- (.4,.7) -- (.5,.6) -- (.6,.7) -- (.8,.5);
		\draw[dashed, draw opacity=.5] (.1,.5) -- (.9,.6);
		\draw[dashed, draw opacity=.5] (.1,.4) -- (.8,.5);
		\filldraw[fill=white, draw opacity =.5] (0.1,.5) circle[radius=0.015];
		\filldraw[fill=white, draw opacity =.5] (0.1,.4) circle[radius=0.015];
		\filldraw[fill=white, draw opacity =.5] (0.9,.6) circle[radius=0.015];
		\filldraw[fill=white, draw opacity =.5] (0.8,.5) circle[radius=0.015];
		\filldraw[blue, draw opacity=0] (.4,.7) circle[radius=0.015];
		\filldraw[blue, draw opacity=0] (.6,.7) circle[radius=0.015];
		\draw[blue] (.5,0) node {\scriptsize $\widetilde{S}$ translated by $a$};
	\end{tikzpicture}
	\qquad \qquad
	\begin{tikzpicture}[scale=3]
		\foreach \x in {0,1}
			\draw[dashed, draw opacity=0.5] (\x,1.1) -- (\x,-0.1);
		\filldraw[red, fill opacity=.25, draw opacity=0] (.1,.5) -- (.5,.9) -- (.9,.6) -- cycle;
		\filldraw[blue, fill opacity=.25, draw opacity=0] (.1,.3) -- (.4,.6) -- (.5,.5) -- (.6,.6) -- (.8,.4) -- cycle;
		\draw[red] (.1,.5) -- (.5,.9) -- (.9,.6);
		\draw[blue] (.1,.3) -- (.4,.6) -- (.5,.5) -- (.6,.6) -- (.8,.4);
		\draw[dashed, draw opacity=.5] (.1,.5) -- (.9,.6);
		\draw[dashed, draw opacity=.5] (.1,.3) -- (.8,.4);
		\filldraw[fill=white, draw opacity =.5] (0.1,.5) circle[radius=0.015];
		\filldraw[fill=white, draw opacity =.5] (0.1,.3) circle[radius=0.015];
		\filldraw[fill=white, draw opacity =.5] (0.9,.6) circle[radius=0.015];
		\filldraw[fill=white, draw opacity =.5] (0.8,.4) circle[radius=0.015];
		\filldraw[blue, draw opacity=0] (.4,.6) circle[radius=0.015];
		\filldraw[blue, draw opacity=0] (.6,.6) circle[radius=0.015];
		\draw[blue] (.5,0) node {\scriptsize $\widetilde{S}$ translated by $b$};
	\end{tikzpicture}
	\caption{The \textcolor{blue}{blue} $\widetilde{S}$'s, translated by $a$ on the left and by $b$ on the right, are lifts of $S$. The \textcolor{red}{red} $\widetilde{S}'$ is a lift of $S'$. On the left, the lifts of the two generatators in $\gen(S)$, marked by blue dots, are $a$-codependent. On the right, the lifts of the two generators are \emph{not} $b$-codependent.}\label{fig:complicated sheets}
	\end{center}\end{figure}
	In the left picture, we have that two elements $y,y'$ of $\gen(S)$ are $a$-codependent but in the right picture, we have that the elements are \emph{not} $b$-codependent.
	For some morphism, choosing $\phi(f_y)$ to be path like with length $a+\upbnd(y)$ determines $\phi(f_{y'})$.
	However, choosing $\phi(f_y)$ to be path like with length $b+\upbnd(y)$ means the choice for $\phi(f_{y'})$ is independent.
	What's happening is that $\codep_{y,a}=\codep_{z,a}$ in the first case but $\codep_{y,a'}\cap\codep_{z,a'}=\emptyset$ in the second case.
	
	We introduce the following definition to check for situations similar to this.
	\begin{definition}\label{def:range of codependence}
		Let $S,S'$ be sheets in $\ReppreI$ and let $\codep_{y,a}$ be as before.
		The \emph{range of codependence of $\codep_{y,a}$} is the set of all $b\in\RR_{\geq a}$ such that, for all $z,w\in \codep_{y,a}$, if $\codep_{z,b}$ and $\codep_{w,b}$ are nonempty then $\codep_{z,b}=\codep_{w,b}$.
	\end{definition}
	Notice that the range of codependence of $\codep_{y,a}$ always includes $a$.
	
	Using this definition we see that, in the situation from Figure~\ref{fig:complicated sheets}, $b$ is outside the range of codependence of $\codep_{y,a}$.

	So, we have the following proposition.
	\begin{proposition}\label{prop:elementary morphisms}
		Assume the hypotheses from Lemma~\ref{lem:codep determines a morphism}.
		Let $\sum_{i=1}^m\lambda_i f_i$ be a sum of path like morphisms in $S'(y)$ such that, for each $1\leq i \leq m$, the length of $f_i$ is $\upbnd(y)+b<\dwnbnd'(y)$, where $b$ is in the range of codependence of $\codep_{y,a}$.
		Then there is a morphism $\phi:S\to S'$ determined by $\phi(f_y)=\sum_{i=1}^m \lambda_i f_i$.
	\end{proposition}
	\begin{proof}
		For each $b$ in the range of codependence of $\codep_{y,a}$ and each $z,w\in\codep_{y,a}$, we know that $\codep_{z,b}=\codep_{w,b}$ if they are both nonempty, by definition.
		
		For each $f_i$, let $b_i=\len(f_i)-\upbnd(y)$.
		Thus, for each $b_i$ and each $f_i$, we can use Lemma~\ref{lem:codep determines a morphism} to construct a morphism $\phi_i:S\to S'$ determined on $C_b(y)$ by $\phi_i(f_y)=f_i$.
		Then the morphism $\phi$ in the present proposition is given by $\phi=\sum_{i=1}^m \lambda_i\phi_i$.
	\end{proof}
	
	Lemma~\ref{lem:codep determines a morphism} and Proposition~\ref{prop:elementary morphisms} give rise to the following definitions.
	\begin{definition}\label{def:elementary morphisms}
		We introduce two definitions together.
		\begin{enumerate}
			\item A morhism $\phi:S\to S'$ in $\ReppreI$, where $S,S'$ are sheets, is called \emph{elementary} if it is of the form in Proposition~\ref{prop:elementary morphisms}.
			\item Let $\{(y_i,a_i)\}\subset \II\times\RR$ be a possibly infinite collection.
				Assume that if $i\neq j$ then $B_{a_i}(y_i)\cap B_{a_j}(y_j)=\emptyset$.
				Choose an elementary morphism $\phi_i$ for each $(y_i,a_i)$.
				Then the morphism $\phi=\sum_i \phi_i$ is called \emph{multi-elementary}.
		\end{enumerate}
	\end{definition}
	
	To the knowledge of the authors, it does not seem possible to construct a morphism between sheets in $\ReppreI$ that is not a sum of multi-elementary morphisms.
	This leads us to the following conjecture.
	
	\begin{conjecture}\label{conj:finite sums of multielementary}
		Any morphism $\phi:S\to S'$ in $\ReppreI$, where $S,S'$ are sheets, is a finite sum of multi-elementary morphisms.
	\end{conjecture}
	
\section{Bricks}\label{sec:bricks}
	In this section we classify the bricks in $\ReppreI$ and relate them to a result in \cite{A22} (see Theorem~\ref{thm:bricks are sawtooth modules from continuous A quivers} in the present paper).
	We continue to consider modules as functors $\preI\to\kVec$, which coincides with left modules, but these are equivalent to right modules since $\preI$ is canonically isomorphic to $\preI^{\text{op}}$ (Remark~\ref{rmk:preI and preIop}).
	
	We first need to define deep modules.
	
	\begin{definition}
		Let $M$ be a module in $\ReppreI$.
		We say $M$ is \emph{deep} if there is a path like $g:x\to x$ in $\preI$, with $x\in\supp M$, such that $M(g)\neq 0$.
	\end{definition}
	
	We also need a special type of endomorphism called a push up endomorphism.
	
	\begin{definition}\label{def:push up morphism}
		Let $\ell\in \RR_{\geq 0}$.
		For any module $M$ in $\ReppreI$, the \emph{push up endomorphism $\psi_\ell\in\End_{\ReppreI}(M)$ of length $\ell$} is the endomorphism where $(\psi_\ell)_x=M(g)$ where $g:x\to x$ is path like with length $\ell$ if such a path like morphism exists in $\preI$ and $g=0$ otherwise.
	\end{definition}
	
	It is straightforward to check that $\psi_\ell$ is a morphism in $\ReppreI$ for any module $M$, but is possibly the $0$ morphism.
	
	\begin{remark}\label{rmk:push ups}
		We make the following remarks about push up morphisms.
		\begin{itemize}
			\item There is a pushup morphism $\psi_\ell\in\End_{\ReppreI}(M)$ with $\ell>0$ if and only if $M$ is deep.
			\item Let $\psi_\ell$ be a push up morphism in $\End_{\ReppreI}(M)$ for some $M$ in $\ReppreI$.
			It follows form the definition that $(\psi_\ell)^n=\psi_{n\ell}$, for any $n\in\NN$.
			In particular, the endomorphism $\psi_0:M\to M$ is the identity morphism.
		\end{itemize}
	\end{remark}
	
	Recall that a (left or right) module $M$ over a ring $R$ is called a \emph{brick} if $\End(M)$, in the category of (left or right) $R$-modules, is a division ring.
	Thus, we say a representation $M$ of $\preI$ is a \emph{brick} if $\End_{\ReppreI}(M)$ is a division ring.
	
	\begin{proposition}\label{prop:deep modules are not bricks}
		If $M$ in $\ReppreI$ is deep then $M$ is not a brick.
	\end{proposition}
	\begin{proof}
		Since $M$ is deep there exists a nonzero $\psi_\ell\in\End_{\ReppreI}(M)$ such that $\ell>0$.
		Then there exists $n\in\NN_{\geq 0}$ such that $n\ell>1$ and so $(\psi_\ell)^n = \psi_{n\ell}=0$.
		Thus, $\End_{\ReppreI}(M)$ is not a division ring and so $M$ is not a brick.
	\end{proof}
	
	In particular, Proposition~\ref{prop:deep modules are not bricks} says that none of the following modules can be bricks: submodules of indecomposable projectives, decorous quotient modules of indecomposable projectives, and sheet modules.
	
	Now that we know what kinds of modules are not bricks, we introduce sawtooth functions.
	These will help us classify bricks.
	
	\begin{definition}\label{def:sawtooth}
		Let $\partial$ be a function as in Corollary~\ref{cor:permuton functions and decorous submodules} and $[a,b]\subset [0,1]$ a closed subinterval.
		We say $\sawtooth$ is a \emph{sawtooth function} if there is some subinterval $I$ of $\ZZ\cup\{\pm\infty\}$ and set $\{x_i\}_{i\in I}\subset[0,1]$ satisfying the following:
		\begin{enumerate}
			\item We have $\min I\neq\max I \in I$.
			\item If $i<j\in I$ then $x_i<x_j$ in $[a,b]$.
			\item We have $x_{\min I}=a$ and $x_{\max I}=b$.
			\item For each $2i,2i+1\in I$, we have $x_{2i+1}-x_{2i} = \partial(x_{2i})-\partial(x_{2i+1})$.
			\item For each $2i,2i-1\in I$, we have $x_{2i}-x_{2i-1} = \partial(x_{2i})-\partial(x_{2i-1})$. 
		\end{enumerate}
	\end{definition}
	Notice the lack of absolute values in the last two items!
	
	See examples of sawtooth function in Figure~\ref{fig:sawtooth}.
	\begin{figure}\begin{center}
	\begin{tikzpicture}[scale=3]
		\foreach \x in {0,1}
			\draw[dashed, draw opacity=0.5] (\x,1.1) -- (\x,-0.1);
		\draw[blue] (0,.4) -- (.2,.6) -- (.4,.4) -- (.6,.6) -- (.8,.4) -- (1,.6);
	\end{tikzpicture}
	\qquad
	\begin{tikzpicture}[scale=3]
		\foreach \x in {0,1}
			\draw[dashed, draw opacity=0.5] (\x,1.1) -- (\x,-0.1);
		\draw[blue] (0,.1) -- (.2,.3) -- (.3,.2) -- (.5,.4) -- (.6,.3) -- (.8,.5);
	\end{tikzpicture}
	\qquad
	\begin{tikzpicture}[scale=3]
		\foreach \x in {0,1}
			\draw[dashed, draw opacity=0.5] (\x,1.1) -- (\x,-0.1);
		\draw[blue] (0.2,0.5) -- (0.4,0.7) -- (0.5,0.6) -- (0.55,0.65) -- (0.575,0.625) -- (0.5875,0.6375) -- (0.59375,.63125) -- (0.596875,0.624475) -- (0.5984375,0.6278625) -- (0.6,.625);
	\end{tikzpicture}
	\caption{Some examples of sawtooth functions from Definition~\ref{def:sawtooth}. The right most example has an accumulation of the sawtooth waves, which is allowed.}\label{fig:sawtooth}
	\end{center}\end{figure}
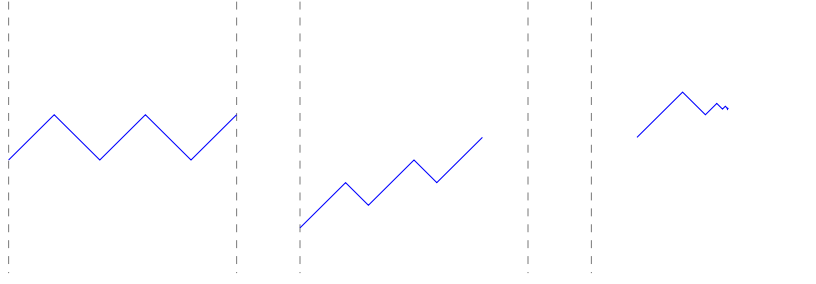
	
	We can use sawtooth functions to define a nice class of modules.
	
	\begin{definition}\label{def:sawtooth module}
		Let $M$ be in $\ReppreI$.
		We say $M$ is a \emph{sawtooth module} if there exists a sawtooth function $\sawtooth$ such that
		\begin{itemize}
			\item the closure of $\supp M$ in $[0,1]$ is $[a,b]$,
			\item $a\in\supp M\Leftrightarrow a=x_i,i\in\ZZ$,
			\item $b\in\supp M\Leftrightarrow b=x_i,i\in\ZZ$, and
			\item if $M(f)\neq 0$ for some path like $f:x\to y$ in $\preI$, then
			\[
				f = \begin{cases}
					\alpha_{yx} & \text{if } x_{2i-1} \leq x \leq y \leq x_{2i} \\
					\alpha^*_{yx} & \text{if } x_{2i} \leq y \leq x \leq x_{2i+1}.
				\end{cases}
			\]
		\end{itemize}
	\end{definition}
	
	We will show that the bricks in $\ReppreI$ are certain kinds of sawtooth modules.
	To do this, we first must prove the following lemma.
	
	\begin{lemma}\label{lem:indecomposable + no push up yields sawtooth}
		Let $M$ be an indecomposable in $\ReppreI$ such that $\supp M$ is not a singleton and $M$ is not deep.
		Then $M$ is a sawtooth module.
	\end{lemma}
	\begin{proof}
		Choose some $x\in \supp M$ such that $x\notin\{\inf\supp M, \lub\supp M\}$.
		Consider the following sets:
		\[
			\text{(i)}\ \{y\in [x,1)\mid M(\alpha_{yx})\neq 0\}
			\qquad \qquad
			\text{(ii)}\ \{y\in [x,1)\mid M(\alpha^*_{xy})\neq 0\}.
		\]
		By assumption on push up functions in $\End_{\ReppreI}(M)$, exactly one of these sets must be $\{x\}$ and the other must contain more than just $x$.
		The technique we will use is symmetric to either case.
		So, for simplicity of writing, we assume the set (i) contains more than just $x$ and (ii) is $\{x\}$.
		Set $A_x$ equal to set (i).
		
		If $\lub A_x\notin A_x$ then we see $\lub A_x=\lub\supp M\notin \supp M$, or else $M$ is not indecomposable.
		Let $y_0=\lub A_x$.
		Suppose $x_0\in A_x$ and $x_0 \lneq \lub\supp M$.
		Then we define $A_{x_0} = \{y\in [x_0,1) \mid M(\alpha^*_{x_0 y}\neq 0\}$, similar to set (ii) above.
		Again, either $\lub A_{x_0}\in A_{x_0}$ or $\lub A_{x_0}=\lub\supp M$.
		
		We proceed inductively where the $A_{x_{2i}}$'s are defined in style (ii) and the $A_{x_{2i+1}}$'s are defined in style (i).
		For each $x_i$, the other type of set must be a singleton, as before.
		If there is a largest $i$ in this process, then $\lub\supp M=x_i$ for some $i\in\ZZ$.
		If, for each $i\geq 0$ in $\ZZ$ there is an $x_{i+1}$, then the $x_i$'s must accumulate to $\lub\supp M$, or else $M$ is not indecomposable.
		Moreover, in this case, $\lub\supp M\notin \supp M$ for the same reason so we set $x_{+\infty}=\lub\supp M$.
		
		We now construct the sets $A^{x},A^{x_{-1}},A^{x_{-2}},\ldots$ where
		\begin{align*}
			A^x &=\{ y\in (0,x] \mid M(\alpha_{xy})\neq 0\} \\
			A^{x_{2i-1}} &= \{y\in (0,x_{2i-1}] \mid M(\alpha^*_{y x_{2i-1}})\neq 0\} \\
			A^{x_{2i}} &= \{y\in (0,x_{-2i}] \mid M(\alpha_{x_{-2i} y})\neq 0\}.
		\end{align*}
		Again, either this process terminates so $\inf\supp M=x_i$ for some $i\in \ZZ$ or $\lim_{i\to-\infty} x_i=\inf\supp M$ and $\inf\supp M\notin\supp M$.
		In the later case, $x_{-\infty}=\inf\supp M$.
	
		We now have enough to define a sawtooth function.
		Let $I\subset\ZZ\cup\{\pm\infty\}$ be the set such that $i\in I$ if $x_i$ is defined above.
		Define $\partial(x_0)=0$.
		Inductively, for each $i\in\ZZ$ such that $i\neq 0$, define
		\[
			\partial(x_i)=\begin{cases}
				\partial(x_{i+1}) + x_i - x_{i+1} & i < 0, i\text{ even} \\
				\partial(x_{i+1}) + x_{i+1} - x_i & i < 0, i\text{ odd} \\
				\partial(x_{i-1}) + x_{i-1} - x_i & i > 0, i\text{ even} \\
				\partial(x_{i-1}) + x_i - x_{i-1} & i > 0, i\text{ odd}.
			\end{cases}
		\]
		If $x_{\pm}$ has been defined, set
		\[
			\partial(x_{\pm\infty}) = \lim_{i\to\pm\infty} \partial(x_i),
		\]
		where the $\pm$'s must match.
		For points $x$ such that $x_i<x<x_{i+1}$, we define $\partial(x)$ to be the point such that $|x-x_i|=|\partial(x)-\partial(x_i)|$ and $|x-x_{i+1}|=|\partial(x)-\partial(x_{i+1})|$.
		Denote by $\overline{\supp M}$ the smallest closed subinterval of $[0,1]$ containing $\supp M$.
		For $x\in [0,1]$ such that $x\notin\overline{\supp M}$, define
		\[
			\partial(x) = \begin{cases}
				\partial(\inf\supp M) & x < \inf\supp M \\
				\partial(\lub\supp M) & x > \lub\supp M.
			\end{cases}
		\]
		It is straightforward to check that $\partial$ is a function as in Corollary~\ref{cor:permuton functions and decorous submodules} and that $\partial|_{\overline{\supp M}}$ is a sawtooth function.
		
		We have shown
		\begin{itemize}
			\item $[a,b]$ is the closure of $\supp M$ in $[0,1]$,
			\item for each $y<z$ in $[x_{2i},x_{2i+1}]$, $M(\alpha_{zy})=0$, and
			\item for each $y<z$ in $[x_{2i-1},x_{2i}]$, $M(\alpha^*_{yz})=0$.
		\end{itemize}
		Therefore, $M$ is a sawtooth module with sawtooth function $\partial|_{\overline{\supp M}}$.
	\end{proof}
	
	An immediate consequence of the lemma is that any brick must be a sawtooth module.
	But we can do better.
	
	Let $\sawtooth$ be a sawtooth function.
	
	For an interval $[a,b]\subset (0,1)$, define the map $\zeta_M:[a,b]\to \RR$ by
	\[
		\zeta(x) = \begin{cases}
		\tan( (x-\frac{a+b}{2})\cdot \frac{\pi}{b-a} ) & x \leq \frac{a+b}{2} \text{ and } (a=x_{-\infty}\text{ or }a=0) \\
		\tan( (x-\frac{a+b}{2})\cdot \frac{\pi}{b-a} ) & x \geq \frac{a+b}{2} \text{ and } (b=x_{+\infty}\text{ or }b=1) \\
		(x-\frac{a+b}{2})\cdot \frac{1}{b-a} & x \leq \frac{a+b}{2} \text{ and } 0<a=x_i, i\in\ZZ \\
		(x-\frac{a+b}{2})\cdot \frac{1}{b-a} & x \geq \frac{a+b}{2} \text{ and } 1>b=x_i, i\in\ZZ 
		\end{cases}
	\]
	We extend $\zeta$ so that we have a bijection in the following way.
	First,
	\begin{align*}
		\hat{a} &= \begin{cases} a & a=x_{-\infty} \\ 0 & a=x_i,i\in\ZZ \end{cases} &
		\hat{b} &= \begin{cases} b & b=x_{+\infty} \\ 1 & b=x_i,i\in\ZZ. \end{cases}
	\end{align*}
	Then we define $\hat{\zeta}:(\hat{a},\hat{b})\to \RR$ by
	\[
		\hat{\zeta}(x) = \begin{cases}
			\zeta(x) & x\in[a,b] \\
			\tan((x-a)\cdot \frac{\pi}{a-\hat{a}})-1 & \hat{a}<x<a \\
			\tan((x-b)\cdot \frac{\pi}{\hat{b}-b})+1 & b<x<\hat{b}.
		\end{cases}
	\]
	If $\hat{a}=a$ then the second line in the definition is not used.
	Similarly, if $\hat{b}=b$ then the third line in the definition is not used.
	If $\hat{a}<a$ then $\zeta(a)=-1$ and $\lim_{x\to a^-}\hat{\zeta}(x)=-1$.
	The similarly statement for $b<\hat{b}$ is also true.
	Thus, $\hat{\zeta}$ is a continuous strictly order preserving bijection from $(\hat{a},\hat{b})$ to $\RR$.
	
	A continuous quiver $Q$ of type $A$ is essentially a pair $(S,\preceq)$ where $S\subset R$ is a discrete set (possibly infinite) and $\preceq$ is a new partial order on $\RR$.
	Let $s,s'\in S$ such that $\not\exists s''\in S$ such that $s<s''<s'$.
	Then the order $\preceq$ on the set $[s,s']\subset\RR$ must be either the same as $\leq_\RR$ or the opposite.
	We think of $Q$ as a category where the objects are $\RR$ and there is a unique morphism $g_{yx}:x\to y$ in $Q$ if $y\preceq x$.
	A representation $M$ of $Q$ is a functor $Q\to \kVec$.
	It is pointwise finite-dimensional if $M(x)$ is a finite-dimensional vector space for each $x\in\RR$.
	By \cite[Theorem 2.3.2]{IRT23}, and independently \cite[Corollary 5.9]{BCB20}, all of the indecomposable representations of $Q$ are interval represenations in the same way the indecomposable representations of type $A_n$ are interval representations.
	By \cite[Theorem 3.0.1]{IRT23} these are also all bricks, just like the $A_n$ case.
	See \cite{IRT23} for more details and the introduction to the study of all representations of continuous quivers of type $A$.
	
	Let $Q$ be a continuous quiver of type $\mathbb{A}$ whose sinks are $\{\zeta(x_{2i})\}$ and whose sources are $\{\zeta(x_{2i+1})\}$.
	
	Then we define a functor $Z: Q\to \preI$ by 
	\begin{align*}
		Z(x) &= \hat{\zeta}^{-1}(x)
		&
		Z(g_{yx}:x\to y) &= \alpha^{(*)}_{\hat{\zeta}^{-1}(y),\hat{\zeta}^{-1}(x)}.
	\end{align*}

	It is quick to check that $Z$ is a functor.
	Then we have the induced functor $Z^*:\ReppreI\to \mathrm{Rep}(Q)$ on modules given by $M\mapsto M\circ Z$ on objects.
	
	\begin{proposition}\label{prop:sawtooth iso on endomorphism rings}
		Let $M$ be a sawtooth module with function $\sawtooth$ and induced functions $\zeta,\hat{\zeta}$.
		Let $Q$ be the corresponding continuous quiver of type $\mathbb{A}$ with functors $Z,Z^*$.
		Then $\End_{\ReppreI}(M)\cong \End_{\Rep(Q)}(Z^*M)$.
	\end{proposition}
	\begin{proof}
		First, note that if $a=\inf\supp M\in \supp M$ then $a=x_i$ for some $i\in\ZZ$ and so $\zeta(\inf\supp M)=-1$.
		The similar statement is true if $b=\lub\supp M\in\supp M$.
		Thus, $\supp M$ is in the domain of $\zeta$.
		
		Let $\phi\in\End_{\ReppreI}(M)$.
		Suppose $Z^*\phi=0$.
		Then, since $\supp Z^*M = \zeta(\supp M)$ we see $\phi=0$.
		
		Now suppose $\bar{\phi}\in \End_{\Rep(Q)}(Z^*M)$ and $\bar{\phi}\neq 0$.
		Then define $\phi:M\to M$ by $\phi_x=\bar{\phi}_{\zeta(x)}$.
		If $x_{2i} \leq y \leq x \leq x_{2i+1}$ with $x,y\in\supp M$ then $M(\alpha^*_{yx})\neq 0$ and $Z^*M(g_{\zeta(y)\zeta(x)})=M(\alpha^*_{yx})$.
		Thus,
		\begin{align*}
			\phi_y M(\alpha^*_{yx}) &= \bar{\phi}_{\zeta(y)}M(g_{\zeta(y)\zeta(x)}) \\
			&= M(g_{\zeta(y)\zeta(x)}) \bar{\phi}_{\zeta(x)} \\
			&= M(\alpha^*_{yx})\phi_x.
		\end{align*}
		Similarly, if $x_{2i-1}\leq x \leq y \leq x_{2i}$ with $x,y\in\supp M$ then $\phi_y M(\alpha_{yx})=M(\alpha_{yx})\phi_x$.
		
		For all other path like $f$ in $\preI$, we know $M(f)=0$ and so trivially $\phi_y M(f)=0=M(f)\phi_x$.
		Thus, $\phi\in\End_{\ReppreI}(M)$ and, by definition, $Z^*\phi=\bar{\phi}$.
		Therefore, $\End_{\ReppreI}(M)\cong \End_{\Rep(Q)}(Z^*M)$.
	\end{proof}
	
	Notice that the bijection in Proposition~\ref{prop:sawtooth iso on endomorphism rings} is induced by $Z^*$.
	
	The relation of our brick modules to the result just after Theorem 0.4 in \cite{A22} comes from the following theorem.
	
	\begin{theorem}\label{thm:bricks are sawtooth modules from continuous A quivers}
		Let $M$ in $\ReppreI$ be a module.
		Then $M$ is a brick if and only $M$ is simple or $M$ is a sawtooth module and the module $Z^*M$ is also a brick in $\Rep(Q)$, where $Q$ and $Z^*$ are induced from the sawtooth function $\sawtooth$ given by $M$.
	\end{theorem}
	\begin{proof}
		The converse direction is trivial for simple modules and trivial for other modules by Proposition~\ref{prop:sawtooth iso on endomorphism rings}.
		
		Now suppose $M$ is a brick.
		If $M$ is simple we are done, so suppose not.
		Since $M$ is a brick, $M$ is indecomposable and $\End_{\ReppreI}(M)\cong\Bbbk$.
		By Proposition~\ref{prop:deep modules are not bricks} we know $M$ is not deep.
		Then, by Lemma~\ref{lem:indecomposable + no push up yields sawtooth} we see $M$ is a sawtooth module.
		So we have the functor $Z^*:\ReppreI\to \Rep(Q)$ from before.
		Then, by Proposition~\ref{prop:sawtooth iso on endomorphism rings} we see that $\End_{\ReppreI}(M)\cong\End_{\Rep(Q)}(Z^*M)\cong \Bbbk$.
	\end{proof}
	
	Finally, we relate bricks back to decorous submodules of representable projectives.
	\begin{proposition}\label{prop:bricks and decorous}
		Let $M$ be a brick in $\ReppreI$ with sawtooth function $\sawtooth$.
		Assume either $x_{\min I}>0$ or $\min I$ is not odd and assume either $x_{\max I}< 1$ or $\max I$ is not odd.
		Then there is a representable projective $P_z=\Hom_{\ReppreI}(z,-)$ and decorous submodule $\widehat{M}\subset P_z$ such that $M$ is a quotient of $\widehat{M}$.
	\end{proposition}
	\begin{proof}
		Let $\widehat{M}$ be the decorous submodule of $P_z$ corresponding to $\partial$, by Corollary~\ref{cor:permuton functions and decorous submodules}.
		If $\min I$ is odd, then $x_{\min I} > 0$ and so $x_{\min I}\in \supp \widehat{M}$ if $x_{\min I}>0$.
		Similarly, $x_{\max I}\in \supp \widehat{M}$ if $x_{\max I}<1$.
		Thus, $\supp M\subset\supp \widehat{M}$.
		
		We now define $\phi:\widehat{M}\to M$.
		For each $x\in (0,1)\setminus\supp M$, set $\phi_x=0$.
		For $x\in\supp M$, define $\phi_x(f_y)=1\in\Bbbk=M(x)$ and define $\phi_x(f)=0$ for all other path like elements in $M(x)$.
		By the Definition~\ref{def:sawtooth module}, we know that $M(f)=0$ if $f$ is not one of $\alpha_{yx}$ for $x_{2i-1}\leq x \leq y \leq x_{2i}$ or $\alpha^*_{yx}$ for $x_{2i} \leq y \leq x \leq x_{2i+1}$.
		In these cases,
		\[
			\phi_y\widehat{M}(\alpha^{(*)}_{yx})(f_x)=\phi_y(f_y)=1=M(\alpha^{(*)}_{yx})(1)=M(\alpha^{(*)}_{yx})\phi_x(f_x).
		\]
		
		For all other path like $h:x\to y$ in $\preI$, we know $\widehat{M}(h)(f_x)\neq f_y$, by construction.
		Thus, $\phi_y\widehat{M}(h)(f_x)=0=M(h)\phi_x(f_x)$.
		Therefore, $\phi:\widehat{M}\to M$ is a morphism in $\ReppreI$ and we see that it is also surjective.
	\end{proof}
	
	If $\min I$ is odd and $x_{\min I}=0$ then the proposition fails.
	This is because $(x_{\min I},x_{\min I+1})\cap\supp\widehat{M}=\emptyset$ but $(x_{\min I},x_{\min I+1})\subset\supp M$.
	Similarly, the proposition fails if $x_{\max I}=1$ and $\max I$ is odd.
	
	It is possible to ``discretize'' the methods in this section to preprojective algebras of type $\mathbb{A}_n$ and recover the result just after Theorem 0.4 in \cite{A22}.

\section{Permutons and ideals in the continuous preprojective algebra}\label{sec:permutons}

Given the result of Mizuno recalled earlier as Theorem \ref{mizuno-thm}, it is natural to wonder if there is a kind of ``continuous permutation'' which defines an ideal in the continuous preprojective algebra. It turns out that the correct notion of continuous permutation is not the naive one (permutations of the points in $(0,1)$), but rather the notion of \emph{permuton}, which was introduced in order to study the limit of a sequence of permutations of increasing size \cite{HK}.

We follow the slightly later \cite{GGKK} (which was also the first paper to use the term ``permuton'') in defining a permuton as a measure $\mu$ on the $\sigma$-algebra of Borel sets of $[0,1]\times [0,1]$ such that
$\mu([0,k]\times[0,1])=k=\mu([0,1]\times[0,k])$.

A permutation $w\in\S_n$ defines a permuton $\gamma_w$ as follows. Divide the unit square up into $n$ equal rows and $n$ equal columns, and then put a uniform measure of weight $1/n$ on each of the squares $(i,w(i))$. General permutons arise as limits in a suitable sense of those coming from permutations.

Here is the permuton for the permutation $w=25341$. (We maintain our convention that $y$ increases downwards.) Each grey square has total measure $1/5$.

$$
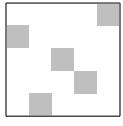
\begin{tikzpicture}[xscale=1.5,yscale=-1.5]
  \foreach \i in {0,5} 
           { \draw (0,\i/5)--(1,\i/5);
             \draw (\i/5,0)--(\i/5,1); }
       \draw [fill=gray, fill opacity=.5, draw opacity=0]   (1/5,2/5) rectangle (0,1/5);
       \draw [fill=gray, fill opacity=.5, draw opacity=0]    (2/5,5/5) rectangle (1/5,4/5);
       \draw [fill=gray, fill opacity=.5, draw opacity=0]    (3/5,3/5) rectangle (2/5,2/5);
        \draw [fill=gray, fill opacity=.5, draw opacity=0]   (4/5,4/5) rectangle (3/5,3/5);
       \draw [fill=gray, fill opacity=.5, draw opacity=0]    (5/5,1/5) rectangle (4/5,0/5);
\end{tikzpicture}$$

We will now show how to associate to an arbitrary permuton $\mu$ an ideal in the continuous preprojective algebra on $(0,1)$.

A function $f$ from $(0,1)$ to $\mathbb R$ satisfying that if $0<x<y<1$, then $|f(y)-f(x)|\leq y-x$ is necessarily continuous, so it has well-defined values at 0 and 1, which we refer to as $f(0)$ and $f(1)$. For $k\in(0,1)$, define a set of functions $\mathcal B_k$ to be those functions satisfying the previous property, such that as well $f(0)=k$ and $f(1)=1-k$.

We now consider right modules, which are equivalent to functors $\preI\to\kVec$ (left modules) since $\preI$ is canonically isomorphic to $\preI^{\text{op}}$ (Remark~\ref{rmk:preI and preIop}).
In particular, the results from Sections~\ref{sec:category:projectives}~and~\ref{sec:subs and quots of projs} apply.

\begin{lemma}\label{lem:decorous from permuton} The functions defining a decorous submodule of $P_k$ are exactly the functions in $\mathcal B_k$. \end{lemma}

\begin{proof} Use Corollary~\ref{cor:permuton functions and decorous submodules}. \end{proof}

For $f\in \mathcal B_k$, let $D_f$ be the submodule of $P_k$ defined by $f$. Similarly, let $U_f=P_k/D_f$.

\begin{lemma} Let $\mu$ be a permuton. For fixed $y$, the function
  $$f_{\mu,y}(x)=-2\mu([0,x]\times[0,y])+y+x$$
  belongs to $\mathcal B_y$.\end{lemma}
\begin{proof} For $x_1\leq x_2$, we have $$\mu([0,x_1]\times[0,y])-\mu([0,x_2]\times[0,y])=\mu((x_1,x_2]\times[0,y)$$ and 
          \begin{equation*}0\leq \mu((x_1,x_2]\times[0,y)\leq \mu((x_1,x_2]\times[0,1])=x_2-x_1.\end{equation*} The remainder of the proof is completely straightforward. \end{proof}

Thus, we can define the permuton ideal:
$$I_\mu = \bigoplus_{y\in(0,1)} D_{f_{\mu,y}}.$$
It is a right ideal of the continuous preprojective algebra, since it is a submodule of the continuous preprojective algebra viewed as a (right) module over itself. In fact, analogously to the discrete case, we have the following:

\begin{proposition}\label{prop:permuton ideal is two sided} For $\mu$ a permuton, $I_\mu$ is a two-sided ideal. \end{proposition}

\begin{proof} We need to check that an element of $I_\mu$, multiplied on the left by an element of the continuous preprojective algebra, still belongs to $I_\mu$. It suffices to check for $a\in I_\mu^q$ and $p<q<r$ that $\alpha^*_{pq}(a)\in I_\mu^p$ and $\alpha_{rq}(a)\in I_\mu^r$. We only consider the case of $\alpha^*_{pq}(a)$; the other is the same. The map induced by $\alpha^*_{pq}$ from $P_q$ to $P_p$ sends the top of $P_q$ to the top point of $P_p$ in the $x=q$ column, which moves $P_q$ down by $q-p$.
  Thereby, $\alpha^*_{pq}$ induces a map from $I_{\mu,q}$ to $P_p$. Since the function defining the bottom of $P_p$ is $1-|x-(1-p)|$, the image of $I_{\mu,q}$ is defined by the function
  $$g(x)=\min(1-|x-(1-p)|, f_{\mu,q}(x)+(q-p)).$$
  We must check $g(x)\geq f_{\mu,p}$. Since $1-|x-(1-p)|$ is the lower boundary of $P_p$, it is clear that $1-|x-(1-p)|\geq f_{\mu,p}$. We also check
  \begin{align*} f_{\mu,q}(x)+(q-p)&=-2\mu([0,x]\times[0,q]) +(q-x) + (q-p)\\
    &=-2\mu([0,x]\times[0,p] + (p-x)  + 2(q-p)-2\mu([0,x]\times (p,q])\\
        &=f_{\mu,p}(x)+2(q-p-\mu([0,x]\times(p,q])\end{align*}

  Since $\mu([0,x]\times(p,q])\leq q-p$, it follows that $g(x)\geq f_{\mu,p}(x)$, as desired.
\end{proof}

\begin{example}
  Let us consider some examples of permuton ideals.
 \begin{enumerate} 
  \item For the permuton $\mu_\id$ which assigns to a set the measure of its intersection with the line $y=x$ (and scaled by $1/\sqrt 2$), the function $f_{\mu_\id,y}(x)=-2\min(x,y)+y+x$. The function $f_{\mu_\id,y}$ describes the top border of $P_y$. Thus $I_{\mu_{\id}}=\Lambda_\Lambda$, where $\Lambda=\Lambda_{(0,1)}$.

\item Similarly, if we take the permuton defined instead using the line $x+y=1$, we determine that the corresponding permuton ideal is the zero ideal.

\item If we take the permuton $\mu_\unif$, defined by the uniform measure, we see that $(I_{\mu_{\unif}})^i$ is the bottom half of $P_i$, cut off by a line segment joining the left and right corners of $P_i$.

\item Finally, if we take the permuton $\gamma_w$ for $w=25341$, the submodules $I_{\gamma_w}^{1/5}$, $I_{\gamma_w}^{2/5}$, $I_{\gamma_w}^{3/5}$, $I_{\gamma_w}^{4/5}$ are as follows:

$$
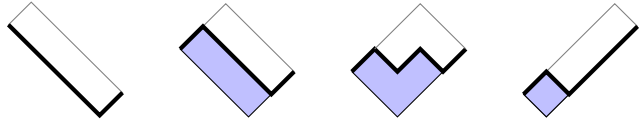
\begin{tikzpicture}[scale=1.5]
    \draw[gray] (0,4/5)--(1/5,1)--(1,1/5)--(4/5,0)--(0,4/5);
    \draw[ultra thick] (0,4/5)--(4/5,0)--(1,1/5);
\end{tikzpicture}
\qquad
\begin{tikzpicture}[scale=1.5]
  \draw[fill=blue, fill opacity=.25] (4/5,1/5)--(3/5,0)--(0,3/5)--(1/5,4/5)--(4/5,1/5);
  \draw[gray] (1/5,4/5) -- (2/5,1)-- (1,2/5)--(4/5,1/5); 
    \draw[ultra thick] (0,3/5)--(1/5,4/5)--(4/5,1/5)--(1,2/5);
\end{tikzpicture}
\qquad
\begin{tikzpicture}[xscale=1.5,yscale=1.5]
  \draw [fill=blue, fill opacity=.25] (0,2/5)--(1/5,3/5) -- (2/5,2/5)--(3/5,3/5)--(4/5,2/5)--(2/5,0)--(0,2/5);
  \draw [ultra thick] (0,2/5)--(1/5,3/5) -- (2/5,2/5)--(3/5,3/5)--(4/5,2/5)--(1,3/5);
  \draw[gray] (1/5,3/5)--(3/5,1)--(1,3/5);
\end{tikzpicture}
\qquad
\begin{tikzpicture}[xscale=-1.5,yscale=1.5]
  \draw[gray] (0,4/5)--(1/5,1)--(1,1/5)--(4/5,0)--(0,4/5);
  \draw[fill=blue, fill opacity=.25] (1,1/5)--(4/5,2/5)--(3/5,1/5)--(4/5,0)--(1,1/5);
  \draw [ultra thick] (0,4/5)--(3/5,1/5)--(4/5,2/5)--(1,1/5);
\end{tikzpicture}
$$

Observe that this is just an unlabelled version of the pictures we saw in Example \ref{radical-example}(3). This is not a coincidence: it will be explained in Section~\ref{recover}. \end{enumerate}\end{example}

\section{Permuton ideals recover permutation ideals} \label{recover}

The goal of this section is to prove the following theorem:

\begin{theorem}\label{finite-is-infinite} For $w\in \S_n$, and $i\in\{1,\dots,n-1\}$, the summand of the permuton ideal $I_{\gamma_w}^{i/n}$ coincides with the picture of the radical filtration of $I_w^i$ defined in Section \ref{permutation-ideals}. \end{theorem}

Before we prove the theorem, we will require some notation and a lemma.

Write $\S_{\langle i\rangle}$ for the subgroup of $\S_n$ generated by the adjacent transpositions other than $s_i$. It is isomorphic to $\S_i\times \S_{n-i}$. For $w\in \S_n$, if we consider the coset $\S_{\langle i\rangle}w$, it has a unique element of minimal length \cite[Corollary 2.4.5]{BB}. Call it $w^{\langle i\rangle}$. It follows that $w$ can be factored as $w=w_{\langle i \rangle}w^{\langle i\rangle}$, with $w_{\langle i \rangle}\in \S_{\langle i \rangle}$ and $\ell(w)=\ell(w_{\langle i\rangle}) + \ell(w^{\langle i \rangle})$ \cite[Proposition 2.4.4]{BB}. (We note that if we were more closely following the notation of \cite{BB} we would write ${}^{\langle i\rangle}w$.)

\begin{example}\label{coset} Consider $w=25341$. The factorization as $w_{\li}w^{\li}$ is obtained by starting from $w=(12)(23)(45)(34)(23)(45)$ and attempting to move adjacent transpositions in $\S_{\li}$ to the left using commutations and braid moves.

  If $i=2$, we obtain the factorization:
  $$ w= (12)(45)(34)\cdot (23)(34)(45).$$
  If $i=3$, we obtain the factorization:
  $$w= (12)(23)(45)\cdot (34)(23)(45).$$

\end{example}

The minimal-length coset representatives of $w\in \S_n$ in fact admit a simple description:

\begin{lemma}[{\cite[Exercise 2.4]{BB}}]\label{reps} The minimal-length coset representative $w^{\langle i\rangle}$ of $w$ is given by rearranging the elements of $1\dots i$ so they are in increasing order (while occupying the same positions) and similarly for $i+1,\dots, n$. \end{lemma}

\begin{example}\label{coset-cont} Continuing Example \ref{coset}, it is easy to confirm that the factorizations of the minimal coset representatives obtained there agree with the statement of the lemma.
\begin{align*} w^{\langle 2\rangle}&=(23)(34)(45)=13452\\
  w^{\langle 3\rangle}&=(34)(23)(45)=14253.\end{align*}
\end{example}

We now demonstrate the relevance of minimal-length coset representatives to permutation ideals.
  
\begin{lemma}\label{use-rep} $(I_w)^i = (I_{w^{\langle i \rangle}})^i$. \end{lemma}

\begin{proof}
  Because $\ell(w)=\ell(w_{\langle i \rangle i})+\ell(w^{\langle i\rangle})$, there is a reduced expression for $w$ which consists of a reduced expression for $w_{\langle i\rangle}$ followed by a reduced expression for $w^{\langle i\rangle}$. Now apply 
  Lemma~\ref{algo} to determine $(I_w)^i$ using this reduced expression for $w$. The only simple in the top of $P_i$ is $S_i$, so the reduced expression for $w_{\langle i\rangle}$ has no effect on $P_i$. Thus, the effect on $P_i$ of $w$ and $w^{\langle i\rangle}$ are the same. The result follows. \end{proof}

\begin{example} Continuing Example \ref{coset-cont}, it is easy to check that
  $(I_{(23)(34)(45)})^2$ agrees with $(I_w)^2$ and $(I_{(34)(23)(45)})^3$ agrees with $(I_w)^3$. \end{example}

The advantage of being able to replace $w$ by $w^{\li}$ is that there is a canonical choice of reduced expression for $w^{\li}$.

\begin{lemma}\label{reduced} If $u$ is a minimal length coset representative with respect to $\S_\li$, then $u$ has a reduced expression
  $$u=(s_is_{i+1}\dots s_{u^{-1}(i)-1})(s_{i-1}\dots s_{u^{-1}(i-1)-1})\dots(s_1\dots s_{u^{-1}(1)-1}).$$
\end{lemma}

\begin{proof}
Note that, by Lemma \ref{reps}, we have $u^{-1}(1)<\dots<u^{-1}(i)$. 



Now, if we let $w$ be the permutation given by the reduced expression in the sstatement of the lemma, it is easy to check that $w(u^{-1}(j))=j$ for $1\leq j\leq i$. Further, for $i+1\leq j\leq n$, let $t$ be the smallest element of $\{1,\dots,i\}$ such that $u^{-1}(j)<u^{-1}(t)$. If there is no such element, set $t=i+1$. Now, it is easy to see that $w(u^{-1}(j))=u^{-1}(j)-(i+1-t)$. Since there are $(i+1)-t$ elements from $\{1,\dots,i\}$ which appear to the right of $u^{-1}(j)$ in $u$, we have that $u(u^{-1}(j))=u^{-1}(j)-(i+1-t)=w(u^{-1}(j))$. Thus $u$ and $w$ agree on all inputs, so $w=u$.

Further, since the number of pairs $1\leq j<k \leq n$ with $u(j)>u(k)$ equals the length of the expression in the statement of the lemma, that expression is therefore reduced. The lemma is proved. \end{proof}
  
We now turn to the proof of Theorem~\ref{finite-is-infinite}.

\begin{proof}[Proof of Theorem~\ref{finite-is-infinite}]
Clearly, it is sufficient to show that  the piecewise-linear curve that defines the top of the picture of the radical filtration of $(I_w)^i$ agrees with the piecewise-linear curve that defines $D_{\mu_w}^i$. 

Let $u=w^{\li}$. Using the reduced expression for $u$ from Lemma \ref{reduced},
and applying Lemma \ref{algo},
$(I_u)^i$ can be obtained from $P_i$ by stripping off simples in the order $((S_i, S_{i+1}, \dots, S_{u^{-1}(i)-1}),(S_{i-1},\dots,S_{u^{-1}(i-1)-1}),\dots,(S_1,\dots,S_{u^{-1}(1)-1}))$. Note that the first parenthesized subsequence removes simples from the top diagonal of slope $-1$, the second subsequence removes simples from the next diagonal down, and so on. By Lemma \ref{use-rep}, this is also $(I_w)^i$.

Let us now consider the piecewise-linear curve describing the top of $(I_w)^i$.  For each interval $(j/n,(j+1)/n)$, it has slope 1 or $-1$. We see that it has slope 1 on interval $(j/n,(j+1)/n)$ if $u^{-1}(j)\in\{1,\dots,i\}$, and slope $-1$ if $u^{-1}(j)\in\{i+1,\dots,n\}$. 

The function $f_{\mu,i/n}(x)$ also has the same slopes. Both these functions are in $\mathcal B_i$, but there is at most one function with these slopes in $\mathcal B_i$.  Thus the two functions coincide.\end{proof}

\section{Inclusion order on permuton ideals}\label{sec:order on permuton ideals}

In this section we consider the inclusion orders on permutation ideals and permuton ideals. 

Let $u,v\in \S_n$. 
\cite[Lemma 6.5]{ORT} says that $I_u\leq I_v$ if and only if $u\geq v$ in Bruhat order on permutations (one of whose equivalent definitions we shall shortly recall).

On the other hand, the following proposition is easy to see:

\begin{proposition}\label{prop:nice Bruhat condition} $I_\mu \subseteq I_\nu$ if and only if $\mu([0,a]\times[0,b]) \leq \nu([0,a]\times[0,b])$ for all $0< a,b<1$. \end{proposition}

By analogy, if the equivalent conditions of the proposition hold, we say that
$\mu \geq \nu$ in ``permuton Bruhat order.'' 

The permuton Bruhat order in fact recovers Bruhat order on permutations. The main result of this section is the following theorem. 

\begin{theorem}\label{bruhat-restricts} For $u,v\in \S_n$, we have that $\gamma_u \leq \gamma_v$ with respect to the permuton Bruhat order if and only if $u\leq v$ with respect to Bruhat order. \end{theorem}

One approach to proving Theorem \ref{bruhat-restricts} would be via Theorem \ref{finite-is-infinite}. For $u,v$ two permutations, if $\gamma_u \geq \gamma_v$, then $D_{\gamma_u} \subseteq D_{\gamma_v}$, and from Theorem \ref{finite-is-infinite}, it follows that $I_u\subseteq I_v$, so $u\geq v$ by \cite[Lemma 6.5]{ORT}, immediately establishing one direction of the theorem. Nonetheless, we give a direct proof, because it is straightforward and seems potentially of independent interest. 

There are several equivalent definitions of Bruhat order. Some of them refer to reduced expressions and/or covering relations, concepts which are not readily applicable to permutons.

To give the definition which we shall use, we need one piece of notation.
For $u\in \S_n$, define $u[i,j]=\{a\in \{1,\dots,i\} \mid u(a)>j\}$.
Now, \cite[Theorem 2.1.5]{BB} says that $u\leq v$ in Bruhat order if and only if $u[i,j]\leq v[i,j]$ for all $1\leq i,j\leq n$.

We can now prove Theorem \ref{bruhat-restricts}.

\begin{proof}[Proof of Theorem \ref{bruhat-restricts}]
  If $\gamma_u\leq \gamma_v$, then clearly $u\leq v$, since $u[i,j]=\gamma_u([0,i/n]\times [j/n,1])$.
  
  For the converse direction, suppose $u\leq v$, and let $a,b\in(0,1)$. We need to check that
  $\gamma_u([0,a]\times[b,1])\leq \gamma_v([0,a]\times[b,1])$. Let $a=i/n+\delta$, $b=j/n-\epsilon$, with $\delta,\epsilon<1/n$. Then
\begin{align*}\gamma_u([0,a]\times[b,1])= &(1-\delta)(1-\epsilon) u[i,j] +
  \delta(1-\epsilon) u[i+1,j] \\ &\qquad + (1-\delta)\epsilon u[i,j-1] + \delta\epsilon u[i+1,j-1].\end{align*}
  Since the same formula holds with $u$ replaced by $v$, the desired inequality follows. 
\end{proof}

\section{Permuton ideals satisfy an analogue of $\tau$-rigidity}\label{sec:permuaton rigidity}

Our goal in this section is to prove an analogue of Mizuno's $\tau$-rigidity result (Theorem \ref{mizuno-rigid}) in the continuous setting. In the setting of modules over the continuous preprojective algebra, the Auslander--Reiten translation is not well-defined, so the definition of $\tau$-rigidity cannot be directly applied. By analogy with the discrete case, in which Lemma \ref{discrete-tau} says that if $M$ is a submodule of $P_i$ then $\tau M\simeq P_i/M$, in the continuous setting we will treat as a definition of $\tau$ for a module contained in an indecomposable projective $P_i$, that $\tau M\simeq P_i/M$.


Let $\mu$ be a permuton. Write $D_\mu^x$ for the summand of the submodule of $P_x$ in the ideal corresponding to $\mu$. Write $U_\mu^x=P_x/D_\mu^x$. 
We will show the following analogue of Mizuno's $\tau$-rigidity result:

\begin{theorem}\label{cont-rigid} For $a,b\in(0,1)$, we have $\Hom(D_\mu^a,U_\mu^b)=0$.
\end{theorem}

This will follow from the following more general proposition.

\begin{proposition}\label{general-hom-vanishing} Let $f,g\in\mathcal B$. Suppose further that $f-g$ is either weakly decreasing or weakly increasing. Then $\Hom(D_f,U_g)=0$. \end{proposition}

\begin{proof}
It suffices to consider only the weakly decreasing case; the weakly increasing case follows by swapping the roles of left and right. So suppose that $f-g$ is weakly decreasing.

Suppose there is some non-zero morphism $\phi$ from $D_f$ to $U_g$. So there is some $a\in (0,1)$ such that $(x_0,f(x_0))\in D_f$ is sent by $\phi$ to a non-zero element in $U_g$, say a linear combination of a finite collection of points $(x_0,z)$ for various values of $z$. All the $(x,z)$ appearing in the linear combination with non-zero coefficients must have $z$ greater than $g(x_0)$. Let $y_0$ be the largest of them.
Assume that $y_0=f(x_0)$; this is without loss of generality, since $D_f$ is invariant under vertical translation of $f$.
Let  $\epsilon=f(x_0)-g(x_0)$. Follow a line of slope 1 from $(x_0,y_0)$ down/left until it first hits the curve $y=g(x)$ at $(x_1,y_1)$, or we reach $x=0$ (in which case we define $x_1=0, y_1=g(0)$. Let the segment from $(a,b)$ to $(x_1,y_1)$ be $\ell_1$.

  All the points on $\ell_1$ (before hitting the curve) must be in the image of $\phi$, since $(x_0,y_0)$ is, so these points in $D_f$ must not lie in the kernel of $\phi$. Now, assuming we have not already reached the line $x=0$, follow a line of slope $-1$ from $(x_1,y_1)$ up/left until hitting $y=f(x)$ at $(x_2,y_2)$ (or stopping hitting the line $x=0$). Let $\ell_2$ be the segment from $(x_1,y_1)$ to $(x_2,y_2)$. 
  Points in $D_f$ strictly above $\ell_2$ have a point on $\ell_1$ in the submodule they generate, so, since the latter points are not in the kernel of $\phi$, neither are the former. Thus, any point on $y=f(x)$ strictly between $(x_2,y_2)$ and $(x_0,y_0)$ cannot lie in the kernel of $\phi$. So long as we do not reach $x=0$, we 
  can repeat the same argument, finding points $(x_3,y_3)$ on $y=g(x)$ and
  $(x_4,y_4)$ on $y=f(x)$, so that all points strictly between $(x_4,y_4)$ and
  $(x_0,y_0)$ on $y=f(x)$ do not lie in the kernel, and so on.

  Since $g\in\mathcal B$, the distance $x_{2i}-x_{2i+1}\geq \frac1{2}(f(x_{2i})-g(x_{2i}))$ unless $x_{2i+1}=0$. Similarly, since $f\in\mathcal B$, $x_{2i+1}-x_{2i+2}\geq \frac1{2}(f(x_{2i+1})-g(x_{2i+1}))$ unless $x_{2i+2}=0$. Since $f-g$ is weakly decreasing,
  $f(x_j)-g(x_j)\geq \epsilon$ for all $j$. Thus, $x_j-x_{j+1}\geq \frac{\epsilon}{2}$. It follows that after a finite number of steps, we reach the line $x=0$. 

  Near $x=0$, though, $D_f$ is supported only at $y$-values close to $f(0)$, while $U_g$ is supported only at $y$-values close to $g(0)$. Since $f-g$ is decreasing, but $f(a)-g(a)=\epsilon$, we have that $f(0)-g(0)\geq \epsilon$. Thus the
  points on $y=f(x)$ near $x=0$ must be in the kernel of $\phi$. This is a contradiction. 
  \end{proof}

\begin{proof}[Proof of Theorem \ref{cont-rigid}] Let $f(x)=2 \mu((0,a)\times(0,x))-x$, $g(x)=2\mu((0,b)\times(0,x))$. We have that $f,g\in\mathcal B$. If $a\leq b$, then $f(x)-g(x)=2\mu((a,b)\times(0,x))$ which is weakly increasing; if $a\geq b$, then $f(x)-g(x)=-2\mu((b,a)\times(0,x))$, which is weakly decreasing. Thus, in either case, we can apply Proposition~\ref{general-hom-vanishing} to conclude. \end{proof}

	\bibliographystyle{alpha}
	\bibliography{a_preprojective_category.bib}{}

\newcommand{\etalchar}[1]{$^{#1}$}
\begin{thebibliography}{HKM{\etalchar{+}}13}

\bibitem[Asa22]{A22}
Sota Asai.
\newblock Bricks over preprojective algebras and join-irreducible elements in
  coxeter groups.
\newblock {\em Journal of Pure and Applied Algebra}, 226:106812, 2022.

\bibitem[ASS06]{ASS}
Ibrahim Assem, Daniel Simson, and Andrzej Skowroński.
\newblock {\em Elements of the Representation Theory of Associative Algebras.
  Volume 1. Techniques of Representation Theory}.
\newblock Cambridge University Press, 2006.

\bibitem[BB05]{BB}
Anders Björner and Francesco Brenti.
\newblock {\em Combinatorics of Coxeter groups}.
\newblock Springer, 2005.

\bibitem[BCB20]{BCB20}
Magnus~Bakke Botnan and William Crawley-Boevey.
\newblock Decomposition of persistence modules.
\newblock {\em Proceedings of the {A}merican mathematical society},
  148:4581--4596, 2020.

\bibitem[ES06]{ES}
Karin Erdman and Andrzej Skowroński.
\newblock The stable calabi--yau dimension of tame symmetric algebras.
\newblock {\em Journal of the Mathematical Society of Japan}, 58, 2006.

\bibitem[GGKK15]{GGKK}
Roman Glebov, Andrzej Grzesik, Tereza Klimo\v{s}ová, and Daniel Král'.
\newblock Finitely forcible graphons and permutons.
\newblock {\em J. Combin. Theory Ser. B}, 110:112--135, 2015.

\bibitem[HKM{\etalchar{+}}13]{HK}
Carlos Hoppen, Yoshiharu Kohayakawa, Carlos~Gustavo Moreira, Balázs Ráth, and
  Rudini Menezes~Sampaio.
\newblock Limits of permutation sequences.
\newblock {\em J. Combin. Theory Ser. B}, 103:93--113, 2013.

\bibitem[IRT23]{IRT23}
Kiyoshi Igusa, Job~Daisie Rock, and Gordana Todorov.
\newblock Continuous quivers of type {$A$} ({I}) {F}oundations.
\newblock {\em Rendiconti del Circolo Matematico di Palermo Series 2},
  72:833--868, 2023.

\bibitem[IT15]{IT15}
Kiyoshi Igusa and Gordana Todorov.
\newblock Continuous cluster categories {I}.
\newblock {\em Algebras and Representation Theory}, 18:65--101, 2015.

\bibitem[Miz14]{Miz}
Yuya Mizuno.
\newblock Classifying $\tau$-tilting modules over preprojective algebras of
  {D}ynkin type.
\newblock {\em Math. Z.}, 277:665--690, 2014.

\bibitem[ORT15]{ORT}
Steffen Oppermann, Idun Reiten, and Hugh Thomas.
\newblock Quotient-closed subcategories of quiver representations.
\newblock {\em Compositio Math.}, 151:568--602, 2015.

\bibitem[Oud15]{Obook}
Steve~{Y.} Oudot.
\newblock {\em Persistence theory : from quiver representations to data
  analysis}.
\newblock American Mathematical Society, Providence, Rhode Island, 2015.

\bibitem[PRY25]{PRY}
Charles Paquette, Job~Daisie Rock, and Emine Y{\i}ld{\i}r{\i}m.
\newblock Categories of generalized thread quivers.
\newblock {\em arXiv:2410.14656v2 [math.RT]}, 2025.

\bibitem[Roc19]{R20}
Job~Daisie Rock.
\newblock Continuous quivers of type {$A$} ({II}) the {A}uslander-{R}eiten
  space.
\newblock {\em arXiv:1910.04140 [math.RT]}, 2019.

\end{thebibliography}
\end{document}